\documentclass[12pt]{amsart}

\usepackage{latexsym,amssymb} 
\usepackage{amsmath,amsfonts,amsthm}
\input xypic

\newtheorem{theorem}{Theorem}[section]
\newtheorem{proposition}[theorem]{Proposition}

\newtheorem{lemma}[theorem]{Lemma}
\newtheorem{example}[theorem]{Example}

\newtheorem{remark}[theorem]{Remark}

\newcommand{\dom}{\mathbf{d}}

\newcommand{\ran}{\mathbf{r}}

\title[Chapter 1]{Primer on inverse semigroups I}
\author{Mark~V.~Lawson}
\address{Department of Mathematics and the
Maxwell Institute for Mathematical Sciences\\
Heriot-Watt University\\
Riccarton\\
Edinburgh~EH14~4AS\\
\texttt{M.V.Lawson@hw.ac.uk}} 

\begin{document}
\maketitle

\section{Introduction}\setcounter{theorem}{0}

Inverse semigroups were introduced in the 1950's by Ehresmann in France, 
Preston in the UK and Wagner in the former Soviet Union as algebraic analogues of pseudogroups of transformations.
We shall show in this chapter that inverse semigroups can in fact be seen as extensions of presheaves of groups by pseudogroups of transformations.

Inverse semigroups can be viewed as generalizations of groups.
Group theory is based on the notion of a symmetry; that is, a structure-preserving bijection.
Underlying group theory is therefore the notion of a bijection.
The set of all bijections from a set $X$ to itself forms a group, $S(X)$, under composition of functions called the symmetric group.
Cayley's theorem tells us that each abstract group is isomorphic to a subgroup of a symmetric group. 
Inverse semigroup theory, on the other hand, is based on the notion of a partial symmetry;
that is, a structure-preserving partial bijection.
Underlying inverse semigroup theory, therefore, is the notion of a partial bijection (or partial permutation).
The set of all partial bijections from $X$ to itself forms a semigroup, $I(X)$, 
under composition of partial functions called the symmetric inverse monoid.
The Wagner-Preston representation theorem tells us that each abstract inverse semigroup is isomorphic to an inverse subsemigroup of a symmetric inverse monoid.
However, symmetric inverse monoids and, by extension, inverse semigroups in general, are endowed with extra structure, as we shall see.\\

\noindent
{\bf Acknowledgement } These notes, and the ones that follow, were written to accompany a course of lectures I gave at the
workshop dedicated to Semigroups and Categories at the
University of Ottawa in 2010.\\

\section{Basic definitions}\setcounter{theorem}{0}

In this section, we shall introduce the rudiments of inverse semigroup theory motivated by the properties of the symmetric inverse monoids.
Such monoids have not only algebraic structure but also a partial order, a compatibility relation and an underlying groupoid structure
all of which can be defined on arbitrary inverse semigroups.

\subsection{The theorem of Wagner and Preston}

A semigroup $S$ is said to be {\em inverse} if for each $s \in S$ there exists a unique element
$s^{-1}$ such that  
$$s = ss^{-1}s \text{ and } s^{-1} = s^{-1}ss^{-1}.$$
Clearly all groups are inverse semigroups.

An {\em idempotent} in a semigroup is an element $e$ such that $e^{2} = e$.
Idempotents play an important role in inverse semigroup because the elements $s^{-1}s$ and $ss^{-1}$ are both idempotents.
The set of idempotents of $S$ is denoted by $E(S)$.
Two special idempotents are the {\em identity} element, if it exists, and the {\em zero} element, if it exists. 
An inverse semigroup with identity is called an {\em inverse monoid}
and an inverse semigroup with zero is called an {\em inverse semigroup with zero}.
An {\em inverse subsemigroup} of an inverse semigroup is a subsemigroup that is also closed under inverses.
If $S$ is an inverse subsemigroup of $T$ and $E(S) = E(T)$ we say that $S$ is a {\em wide} inverse subsemigroup of $T$.

The symmetric inverse monoid really is an inverse monoid in the terms of this definition.
The only idempotents in $I(X)$ are the identity functions on the subsets of $X$;
that is, partial functions of the form $1_{A}$ where $A \subseteq X$ and $1_{A}$ is the identity function on $A$.

\begin{remark}{\em  The distinction between semigroups and monoids is not a trivial one.
A comparison with $C^{\ast}$-algebras will make the point.
Commutative $C^{\ast}$-algebras correspond to locally compact spaces whereas the
commutative $C^{\ast}$-algebras with identity correspond to compact spaces.}
\end{remark}

A semigroup $S$ is said to be {\em regular} if for each $a \in S$ there exists an element $b$ such that $a = aba$ and $b = bab$.
The element $b$ is said to be {\em an inverse} of $a$.
Thus inverse semigroups are the regular semigroups in which each element has a unique inverse.
The following result is elementary but fundamental.
It was proved independently by Liber in the former Soviet Union,
and Douglas Munn and Roger Penrose over lunch  in St John's College, Cambridge as graduate students.

\begin{proposition} 
A regular semigroup is inverse if and only if its idempotents commute.
\end{proposition}
\begin{proof}
Let $S$ be a regular semigroup 
in which the idempotents commute and  
let $u$ and $v$ be inverses of $x$.  
Then
$$u = uxu = u(xvx)u = (ux)(vx)u,$$
where both $ux$ and $vx$ are idempotents.  
Thus, since idempotents commute, we have that
$$u = (vx)(ux)u = vxu = (vxv)xu = v(xv)(xu).$$
Again, $xv$ and $xu$ are idempotents and so
$$u = v(xu)(xv) = v(xux)v = vxv = v. $$ 
Hence $u = v$.

The converse is a little trickier.
Observe first that in a regular semigroup 
the product of two idempotents $e$ and $f$ has an idempotent inverse.
To see why, let $x = (ef)'$ be any inverse of $ef$.  
Then the element $fxe$ is an idempotent inverse of $ef$.  

Now let $S$ be a semigroup in which every element has a unique inverse.  
We shall show that $ef=fe$ for any idempotents $e$ and $f$.    
By the result above, 
$f(ef)'e$ is an idempotent inverse of $ef$.
Thus $(ef)' = f(ef)'e$ by uniqueness of inverses,
and so $(ef)'$ is an idempotent.
Every idempotent is self-inverse,
but on the other hand, the inverse of $(ef)'$ is $ef$.
Thus $ef = (ef)'$ by uniqueness of inverses.
Hence $ef$ is an idempotent.
We have shown that the set of idempotents  
is closed under multiplication.
It follows that $fe$ is also an idempotent. 
But $ef(fe)ef = (ef)(ef) = ef$, and $fe(ef)fe = fe$
since $ef$ and $fe$ are idempotents.
Thus $fe$ and $ef$ are inverses of $ef$.
Hence $ef = fe$.
\end{proof}

In the symmetric inverse monoid, the product of the idempotents $1_{A}$ and $1_{B}$ is just $1_{A \cap B}$
and so the commutativity of idempotent multiplication is just a reflection of the fact that the intersection of subsets  is commutative.

Inverses in inverse semigroups behave much like inverses in groups.

\begin{lemma} \mbox{}
\begin{enumerate}

\item $(s^{-1})^{-1} = s$.

\item $(st)^{-1} = t^{-1}s^{-1}$.

\item If $e$ is an idempotent then $ses^{-1}$ is an idempotent.

\end{enumerate}
\end{lemma}

We now characterize the two extreme types of inverse semigroup: those having exactly one idempotent
and those consisting of nothing but idempotents.

\begin{proposition}
All groups are inverse semigroups, and an inverse semigroup is a group if and only if it has a unique idempotent.
\end{proposition}
\begin{proof}
Clearly, groups are inverse semigroups.
Conversely, let $S$ be an inverse semigroup with exactly one
idempotent, $e$ say. 
Then $s^{-1}s = e = ss^{-1}$ for each $s \in S$. 
But $es = (ss^{-1})s = s = s(s^{-1}s) = se$, and so 
$e$ is the identity of $S$.
Hence $S$ is a group. 
\end{proof}

Groups are therefore degenerate inverse semigroups.

The following result leads to the set of idempotents of an inverse semigroup being referred to as its {\em semilattice of idempotents}.

\begin{proposition} \mbox{}
\begin{enumerate}

\item Let $S$ be an inverse semigroup.
Then $E(S)$ is a meet semilattice when we define $e \wedge f = ef$.

\item  All meet semilattices are inverse semigroups, and an inverse in which every element is an idempotent is a meet semilattice.

\end{enumerate}
\end{proposition}
\begin{proof} 
(1) Define $e \leq f$ by $e = ef = fe$.
Then this is a partial order on $E(S)$,
and with respect to this order each pair of idempotents $e$ and $f$ has a greatest lower bound $ef$.

(2) Let $(P,\wedge)$ be a meet semilattice. 
Then $P$ is a commutative semigroup in which $e = e \wedge e$ for each element $e \in P$. 
Thus $(P,\wedge)$ is an inverse semigroup in which every element is idempotent.
\end{proof}

In the case of the symmetric inverse monoid $I(X)$, result (1) above is just the fact that the semilattice of
idempotents of $I(X)$ is isomorphic to the boolean algebra of all subsets of $X$.

The following property is often used to show that definitions involving idempotents
are self-dual with respect to left and right.
It is part of the folklore of the subject but it played an interesting,
and rather unexpected role, in Girard's work on linear logic.

\begin{lemma} Let $S$ be an inverse semigroup.
\begin{enumerate}

\item For each idempotent $e$ and element $s$ there is an idempotent $f$ such that
$es = sf$. 

\item For each idempotent $e$ and element $s$ there is an idempotent $f$ such that
$se = fs$. 

\end{enumerate}
\end{lemma}
\begin{proof} We prove (1) only since the proof of (2) is similar.
Put $f = s^{-1}es$ an idempotent.
Then $sf = s(s^{-1}es) = (ss^{-1})es = e(ss^{-1})s = es$,
using the fact that idempotents commute.
\end{proof}

{\em Homomorphisms of inverse semigroups} are just semigroup homomorphisms.
The convention we shall follow is that if $S$ and $T$ are both monoids
or both inverse semigroups with zero then their homomorphisms
will be required to be monoid homomorphisms or map zeros to zeros, respectively.
{\em Isomorphisms of inverse semigroups} are just semigroup isomorphisms.

\begin{lemma} Let $\theta \colon S \rightarrow  T$ be a homomorphism  between inverse semigroups.  
\begin{enumerate}

\item $\theta (s^{-1}) = \theta (s)^{-1}$ for all $s \in S$. 

\item If $e$ is an idempotent then $\theta (e)$ is an idempotent. 

\item If $\theta (s)$ is an idempotent then there is 
an idempotent $e$ in $S$
such that $\theta (s) = \theta (e)$. 

\item $\mbox{\rm Im}\,\theta$ is an inverse subsemigroup of $T$. 

\item If $U$ is an inverse subsemigroup of $T$
then $\theta^{-1}(U)$ is an inverse subsemigroup of $S$.

\end{enumerate}
\end{lemma}
\begin{proof}
(1) Clearly, 
$\theta (s) \theta (s^{-1}) \theta (s) = \theta (s)$ and 
$\theta (s^{-1}) \theta (s) \theta (s^{-1}) = \theta (s^{-1})$.   
Thus by uniqueness of inverses we have that 
$\theta (s^{-1}) = \theta (s)^{-1}$.

(2) $\theta (e)^{2} = \theta (e) \theta (e) = \theta (e)$. 

(3) If $\theta (s)^{2} = \theta (s)$, then 
$\theta (s^{-1}s) = \theta (s^{-1}) \theta (s) 
= \theta (s)^{-1} \theta (s) = \theta (s)^{2} = \theta (s)$.

(4) Since $\theta$ is a semigroup homomorphism $\mbox{im}\,\theta$ 
is a subsemigroup of $T$.
By (1), $\mbox{im}\,\theta$ is closed under inverses.

(5) Straightforward.

\end{proof}

If $\theta \colon S \rightarrow T$ is a homomorphism between inverse semigroups
then it induces a homomorphism between the semilattices $E(S)$ and $E(T)$.
If this restricted homomorphism is injective we say that the homomorphism is {\em idempotent-separating}.

The following result confirms that inverse semigroups are the right abstract counterparts of the symmetric inverse monoids.

\begin{theorem}[Wagner-Preston representation theorem]
Every inverse semigroup can be embedded in a symmetric inverse monoid.
\end{theorem}
\begin{proof} Given an inverse semigroup $S$ we shall construct an injective homomorphism $\theta \colon S \rightarrow I(S)$.
For each element $a \in  S$, 
define $\theta_{a} \colon a^{-1}aS \rightarrow  aa^{-1}S$ 
by $\theta _{a}(x) = ax$.  
This is well-defined because $aS = aa^{-1}S$ as the following set inclusions show
$$aS = aa^{-1}aS \subseteq aa^{-1}S \subseteq aS.$$  
Also
$\theta_{a^{-1}} \colon aa^{-1}S \rightarrow a^{-1}aS$
and 
$\theta_{a^{-1}} \theta_{a}$ is the identity on $a^{-1}aS$
and $\theta_{a} \theta_{a^{-1}}$ is the identity on $aa^{-1}S$.
Thus $\theta _{a}$ is a bijection and   
$\theta_{a}^{-1} = \theta_{a^{-1}}$.
Define $\theta \colon S \rightarrow I(S)$ 
by $\theta (a) = \theta_{a}$.
This is well-defined by the above.
Next we show that $\theta_{a} \theta_{b} = \theta_{ab}$.
If $e$ and $f$ are any idempotents then
$$eS \cap fS = efS.$$
Thus
$$\mbox{dom}\,\theta_{a} \cap \mbox{im}\,\theta_{b} 
= a^{-1}aS \cap bb^{-1}S = a^{-1}abb^{-1}S.$$
Hence 
$$\mbox{dom}(\theta_{a}  \theta_{b}) 
= \theta_{b}^{-1}(a^{-1}abb^{-1}S) = b^{-1}a^{-1}aS
= b^{-1}a^{-1}abS$$
where we use the following subset inclusions
$$b^{-1}a^{-1}aS = b^{-1}bb^{-1}a^{-1}aS = b^{-1}a^{-1}abb^{-1}S \subseteq b^{-1}a^{-1}abS \subseteq b^{-1}a^{-1}aS.$$
Thus $\mbox{dom}(\theta_{a} \theta_{b}) = \mbox{dom}(\theta_{ab})$.
It is immediate from the definitions that 
$\theta_{a} \theta_{b}$ and $\theta_{ab}$ 
have the same effect on elements,
and so $\theta$ is a homomorphism.
It remains to prove that $\theta$ is injective.
Suppose that $\theta_{a} = \theta_{b}$.
Then $a = ba^{-1}a$ and $b = ab^{-1}b$ from which $a = b$ readily follows.
\end{proof}

\begin{example}
{\em 
Let $X$ be a topological space.
Consider the collection $\Gamma (X)$ of all homeomorphisms between the open subsets of $X$.
This is not merely a subset of $I(X)$ but also an inverse subsemigroup.
It is known as a {\em pseudogroup of transformations}.
Admittedly, in many applications the word `pseudogroup' often implies extra properties that will not concern us here.
Pseudogroups of smooth maps between the open subsets of $\mathbb{R}^{n}$ are used to define differential manifolds.
This and similar applications led Ehresmann and Wagner to develop a general theory of pseudogroups with a view to using them in the foundations of differential geometry.
}
\end{example}

\subsection{The natural partial order}

In the previous section, we dealt with the algebraic structures on the symmetric inverse monoid: the product and the inverse.
But the symmetric inverse monoid $I(X)$ has other structures in addition to its algebraic ones,
and these will leave a trace in arbitrary inverse semigroups via the Wagner-Preston representation theorem.

There is a partial ordering on partial bijections called the restriction ordering.
Perhaps surprisingly, this order can be characterized algebraically:
namely, $f \subseteq g$ if and only if $f = gf^{-1}f$.
This motivates our next definition.

On an inverse semigroup, define $s \leq t$ iff $s = ts^{-1}s$. 

\begin{lemma} The following are equivalent.
\begin{enumerate}

\item $s \leq t$.

\item $s = te$ for some idempotent $e$.

\item $s = ft$ for some idempotent $f$.

\item $s = ss^{-1}t$.

\end{enumerate}
\end{lemma}
\begin{proof} 

(1)$\Rightarrow$(2). This is immediate.

(2)$\Rightarrow$(3). This is immediate by Lemma~2.6.

(3)$\Rightarrow$(4). Suppose that $s = ft$.
Then $fs = s$ and so $fss^{-1} = ss^{-1}$.
It follows that $s = ss^{-1}t$.

(4)$\Rightarrow$(1). Suppose that $s = ss^{-1}t$.
Then $s = t(t^{-1}ss^{-1}t)$.
Put $i = t^{-1}ss^{-1}t$.
Then $si = s$ and so $s^{-1}si = s^{-1}s$.
It follows that $s = ts^{-1}s$ giving $s \leq t$.
\end{proof}

We may now establish the main properties of the relation $\leq$.
They are all straightforward to prove in the light of the above lemma.

\begin{proposition} \mbox{}
\begin{enumerate}

\item The relation $\leq$ is a partial order.

\item If $s \leq t$ then $s^{-1} \leq t^{-1}$.

\item If $s_{1} \leq t_{1}$ and $s_{2} \leq t_{2}$ then $s_{1}s_{2} \leq t_{1}t_{2}$.

\item If $e$ and $f$ are idempotents then $e \leq f$ if and only if $e = ef = fe$.

\item $s \leq e$ where $e$ is an idempotent implies that $e$ is an idempotent.

\end{enumerate}
\end{proposition}

\begin{remark}{\em 
Property (1) above leads us to dub $\leq$ the {\em natural partial order} on $S$.
Property (2) needs to be highlighted since readers familiar with lattice-ordered groups
might have been expecting something different.
Property (3) tells us that the natural partial order is {\em compatible with the multiplication}.
Property (4) tells us that when the natural partial order is restricted to the semilattice of idempotents
we get back the usual ordering on the idempotents.
Because the natural partial order is defined algebraically it is preserved by homomorphisms.
}
\end{remark}

Our next result tells us that the partial order encodes how far from being a group an inverse semigroup is.

\begin{proposition} 
An inverse semigroup is a group if and only if the natural partial order is the equality relation.
\end{proposition}
\begin{proof} Let $S$ be an inverse semigroup in which the natural partial order is equality.
If $e$ and $f$ are any two idempotents then $ef \leq e,f$ and so $e = f$.
It follows that there is exactly one idempotent and so $S$ is a group by Proposition~2.4.
The converse is immediate.
\end{proof}

In any poset $(X,\leq)$, a subset $Y \subseteq X$ is said to be an {\em order ideal}
if $x \leq y \in Y$ implies that $x \in Y$.
More generally, if $Y$ is any subset of $X$ then define
$$Y^{\downarrow} = \{x \in X \colon x \leq y \text{ for some } y \in Y \}.$$
This is the {\em order ideal generated by $Y$}.
If $y \in X$ then we denote $\{y \}^{\downarrow}$ by $y^{\downarrow}$
and call it the {\em principal order ideal generated by $y$}.

Property (5) of Proposition~2.11 tells us that the semilattice of idempotents is an order ideal in $S$
with respect to the natural partial order.

Looking below an idempotent we see only idempotents, what happens if we look up?
The answer is that we don't necessarily see only idempotents.
The symmetric inverse monoid is an example.

Let $(X,\leq)$ be a poset.
If $Y$ is any subset of $X$ then define 
$$Y^{\uparrow} = \{x \in X \colon x \geq y \text{ for some } y \in Y \}.$$
If $Y = \{ y\}$ we denote $\{ y\}^{\uparrow}$ by $y^{\uparrow}$.

An inverse semigroup $S$ is said to be {\em $E$-unitary} if $e \leq s$ where $e$ is an idempotent implies that $s$ is an idempotent.
An inverse semigroup with zero $S$ is said to be $E^{\ast}$-unitary if 
$0 \neq e \leq s$ where $e$ is an idempotent implies that $s$ is an idempotent.

\begin{remark}{\em 
The reason for having two definitions, depending on whether the inverse semigroup does not or does have a zero,
is because an $E$-unitary inverse semigroup with zero has to be a semilattice since every element is above the zero.
Thus the definition of an $E$-unitary inverse semigroup in the presence of a zero is uninteresting.
This bifurcation between inverse semigroups-without-zero and inverse semigroups-with-zero permeates the subject.}
\end{remark}

\subsection{The compatibility relation}

As a partially ordered set $I(X)$ has further properties.
The meet of any two partial bijections always exists, but joins are a different matter.
Given two partial bijections their union is not always another partial bijection;
to be so the partial bijections must satisfy a condition that forms the basis of our next definition.

Define $s \sim t$ iff $s^{-1}t,st^{-1} \in E(S)$.
This is called the {\em compatibility relation}.
It is reflexive and symmetric but not generally transitive.

\begin{lemma} A pair of elements bounded above is compatible.
\end{lemma}
\begin{proof} Let $s,t \leq u$.
Then $s^{-1}t \leq u^{-1}u$ and $st^{-1} \leq uu^{-1}$ so that $s \sim t$.
\end{proof}

A subset of an inverse semigroup is said to be {\em compatible} if the elements are pairwise compatible.
If a compatible subset has a least upper bound it is said to have a {\em join}.

\begin{lemma} 
$s \sim t$ if and only if $s \wedge t$ exists and 
$\dom (s \wedge t) = \dom (s) \wedge \dom (t)$
and
$\ran (s \wedge t) = \ran (s) \wedge \ran (t)$.
\end{lemma}
\begin{proof}
We prove that $st^{-1}$ is an idempotent
if and only if 
the greatest  lower bound $s \wedge  t$ of $s$ and $t$ exists and 
$(s \wedge  t)^{-1}(s \wedge  t) = s^{-1}st^{-1}t$.  
The full result then follows by the dual argument.
Suppose that $st^{-1}$ is an idempotent. 
Put $z = st^{-1}t$.  
Then $z \leq  s$ and $z \leq  t$, since $st^{-1}$ is an idempotent.  
Let $w \leq  s,t$.  
Then $w^{-1}w \leq  t^{-1}t$ and so $w \leq  st^{-1}t = z$.  
Hence $z = s \wedge  t$.  
Also
$$z^{-1}z = (st^{-1}t)^{-1}(st^{-1}t) 
= t^{-1}ts^{-1}st^{-1}t = s^{-1}st^{-1}t.$$

Conversely, suppose that $s \wedge t$ exists and 
$(s \wedge  t)^{-1}(s \wedge  t) = s^{-1}st^{-1}t$.
Put $z = s \wedge t$.
Then $z = sz^{-1}z$ and $z = tz^{-1}z$.
Thus $sz^{-1}z = tz^{-1}z$, and so
$st^{-1}t = ts^{-1}s$.
Hence $st^{-1} = ts^{-1}st^{-1}$, which is an idempotent.
\end{proof}

Since the compatibility relation is not always transitive it is natural to ask when it is.
The answer might have been uninteresting but turns out not to be.

\begin{proposition} 
The compatibility relation is transitive if and only if the semigroup is $E$-unitary.
\end{proposition}
\begin{proof}
Suppose that $\sim$ is transitive.  
Let $e \leq  s$, where $e$  is  an  idempotent.  
Then $se^{-1}$ is an idempotent because $e = se = se^{-1}$, 
and $s^{-1}e$ is an idempotent because $s^{-1}e \leq  s^{-1}s$. 
Thus $s \sim e$.  
Clearly $e \sim s^{-1}s$, and so, by our assumption that
the compatibility relation is transitive,
we have that $s \sim s^{-1}s$.  
But $s(s^{-1}s)^{-1} = s$, so that $s$ is an idempotent.

Conversely, suppose that $S$ is $E$-unitary 
and that $s \sim t$ and $t \sim u$.
Clearly $(s^{-1}t)(t^{-1}u)$ is an idempotent and
$$(s^{-1}t)(t^{-1}u) = s^{-1}(tt^{-1})u \leq s^{-1}u.$$
But $S$ is $E$-unitary and so $s^{-1}u$ is an idempotent.
Similarly, $su^{-1}$ is an idempotent.
Hence $s \sim u$.
\end{proof}

The $E^{\ast}$-unitary semigroups also enjoy a property that is more significant than it looks.

\begin{proposition} 
An $E^{\ast}$-unitary inverse semigroups has meets of all pairs of elements.
\end{proposition}
\begin{proof}
Let $s$ and $t$ be any pair of elements.
Suppose that there exists a non-zero element $u$ such that $u \leq s,t$.
Then $uu^{-1} \leq st^{-1}$ and $uu^{-1}$ is a non-zero idempotent.
Thus $st^{-1}$ is an idempotent.
Similarly $s^{-1}t$ is an idempotent.
It follows that $s \wedge t$ exists by Lemma~2.16.
If the only element below $s$ and $t$ is $0$ then $s \wedge t = 0$.
\end{proof}

In an inverse semigroup with zero there is a refinement of the compatibility relation which is important.
Define $s \perp t$ iff $s^{-1}t = 0 = st^{-1}$.
This is the {\em orthogonality relation}. 
If an orthogonal subset has a least upper bound then it is said to have an {\em orthogonal join}.

In the symmetric inverse monoid the union of compatible partial bijections is another partial bijection
and the union of an orthogonal pair of partial bijections is another partial bijection which is a disjoint union.

Inverse semigroups generalize groups: the single identity of a group is expanded into a semilattice of idempotents.
It is possible to go in the opposite direction and contract an inverse semigroup to a group.
On an inverse semigroup $S$ define the relation $\sigma $ by 
$$s \, \sigma \, t \Leftrightarrow \exists u \leq  s,t$$
for all $s,t \in S$.

\begin{theorem} Let $S$ be an inverse semigroup. 
\begin{enumerate}

\item $\sigma $ is the smallest congruence on $S$ containing 
the compatibility relation. 

\item $S/\sigma $ is a group. 

\item If $\rho $ is any congruence on $S$ such that $S/\rho $ 
is a group then $\sigma  \subseteq  \rho $. 

\end{enumerate}
\end{theorem}
\begin{proof}
(1)
We begin by showing that $\sigma $ is an equivalence relation.  
Reflexivity and symmetry are immediate.  
To prove transitivity, let $(a,b),(b,c) \in \sigma$. 
Then there exist elements
$u,v \in S$ such that $u \leq a,b$ and $v \leq b,c$.
Thus $u,v \leq b$.
The set $b^{\downarrow}$ is a compatible subset and so $u \wedge v$ exists
by Lemma~2.15 and Lemma~2.16.
But $u \wedge v \leq a,c$ and so $(a,c) \in \sigma$.
The fact that $\sigma$ is a congruence 
follows from the fact that the natural 
partial order is compatible with the multiplication.  
If $s \sim t$ then by Lemma~2.16, the meet $s \wedge t$ exists.
Thus $s \sigma t$.
It follows that the compatibility relation is contained in the minimum group congruence.

Let $\rho $ be any congruence containing $\sim$, 
and let $ (a,b) \in  \sigma $.  
Then $z \leq  a,b$ for some $z$.  
Thus $z \sim a$ and $z \sim b$.  
By assumption $(z,a),(z,b) \in  \rho $.  
But $\rho $ is an equivalence and so $(a,b) \in  \rho $.  
Thus $\sigma  \subseteq  \rho $. 
This shows that $\sigma$ is the minimum group congruence.

(2) Clearly, all idempotents are contained in a single 
$\sigma $-class (possibly with non-idempotent elements).  
Consequently, $S/\sigma $ is an inverse semigroup with a single idempotent.  
Thus $S/\sigma $ is a group by Proposition~2.4.

(3) Let $\rho $ be any congruence such that $S/\rho $ is a group.  
Let $(a,b) \in  \sigma $.  
Then $z \leq  a,b$ for some $z$.  
Hence $\rho (z) \leq  \rho (a),\rho (b)$.  
But $S/\rho $ is a group and so its natural partial order is equality.  
Hence $\rho (a) = \rho (b)$.
\end{proof}

The congruence $\sigma$ is called the {\em minimum group congruence}
and the group $S/\sigma$ the {\em maximum group image} of $S$.
The properties of this congruence lead naturally to
the following result on the category of inverse semigroups.

\begin{theorem} The category of groups is a reflective subcategory of
the category of inverse semigroups.
\end{theorem}
\begin{proof}
Let $S$ be an inverse semigroup and 
$\sigma^{\natural} \colon \: S \rightarrow S/\sigma$ the natural
homomorphism.
Let $\theta \colon \: S \rightarrow G$ be a homomorphism to a group $G$.
Then $\mbox{\rm ker}\,\theta$ is a group congruence on $S$ and so
$\sigma \subseteq \mbox{\rm ker}\,\theta$ by Theorem~2.19.
Thus by standard semigroup theory there is a unique homomorphism $\theta^{\ast}$ from $S/\sigma$ to $G$ 
such that $\theta = \theta^{\ast} \sigma^{\natural}$.
\end{proof}

It follows by standard category theory, such as Chapter~IV, Section~3 of \cite{Maclane},
that there is a functor from the category of 
inverse semigroups to the category of groups which takes each 
inverse semigroup $S$ to $S/\sigma$. 
If $\theta \colon S \rightarrow  T$ is a homomorphism 
of inverse semigroups then the function  
$\psi \colon  S/\sigma  \rightarrow  T/\sigma $  
defined by $\psi(\sigma (s)) = \sigma (\theta (s))$  
is the corresponding group homomorphism (this can be checked directly).

For inverse semigroups with zero the minimum group congruence is not very interesting
since the group degenerates to the trivial group.
In this case, replacements have to be found.

\begin{remark}
{\em Constructing groups from inverse semigroups might seem a retrograde step
but some important groups arise most naturally as maximum group images of inverse semigroups.}
\end{remark}

\subsection{The underlying groupoid}

The product we have defined on the symmetric inverse monoid $I(X)$ is not the only one nor perhaps even the most obvious.
Given partial bijections $f$ and $g$ we might also want to define $fg$ only when the domain of $f$ is equal to the range of $g$.
When we do this we are regarding $f$ and $g$ as being functions rather than partial functions.
With respect to this restricted product $I(X)$ becomes a groupoid.
A  groupoid is a (small) category in which every arrow is an isomorphism.
Groupoids can be viewed as generalizations of both groups and equivalence relations.
We now review the basics of groupoid theory we shall need.

Categories are usually regarded 
as categories of structures with morphisms.
But they can also be regarded as algebraic structures
no different from groups, rings and fields except that the binary operation
is only partially defined.
We define categories from this purely algebraic point of view.

Let $C$ be a set equipped with a partial binary operation 
which we shall denote by $\cdot$ or by concatenation.  
If $x,y \in C$ and the product $x \cdot y$ is defined we write 
$\exists x \cdot y$.  
An element $e \in C$ is called an {\em identity}
if 
$\exists e \cdot x$ implies $e \cdot x=x$ and $\exists x \cdot e$
implies $x \cdot e=x$. 
The set of identities of $C$ is denoted $C_{o}$; 
the subscript `o' stands for `object'.
The pair $(C, \cdot )$ is said to be a {\em category} 
if the following 
axioms hold:  

\begin{description}

\item[{\rm (C1)}] $x \cdot (y \cdot z)$ exists 
if, and only if, $(x \cdot y) \cdot z$ exists, in which case they are equal. 

\item[{\rm (C2)}] $x \cdot (y \cdot z)$ exists if, and only if, 
$x \cdot y$ and $y \cdot z$ exist. 

\item[{\rm (C3)}] For each $x \in C$ there exist identities 
$e$ and $f$ such that $\exists x \cdot e$ and $\exists f \cdot x$. 

\end{description}

From axiom (C3), it follows that the identities $e$ and $f$ 
are uniquely determined by $x$. 
We write $e = {\bf d} (x)$ and $f = {\bf r} (x)$, where 
${\bf d}(x)$ is the {\em domain} identity and ${\bf r}(x)$ 
is the {\em range} identity.
Observe that $\exists x \cdot y$ if, 
and only if, ${\bf d} (x) = {\bf r} (y)$. 

The elements of a category are called {\em arrows}.
If $C$ is a category and $e$ and $f$ identities 
in $C$ then we put 
$$\mbox{\rm hom}(e,f) 
= \{x \in C \colon \: {\bf d}(x) = e \mbox{ and } {\bf r}(x) = f\},$$
the set of {\em arrows from $e$ to $f$}.
Subsets of $C$ of the form $\mbox{\rm hom}(e,f)$
are called {\em hom-sets}.\index{hom-set}
We also put $\mbox{\rm end}(e) = \mbox{\rm hom}(e,e)$,
the {\em local monoid at $e$}.
A category $C$ is said to be a {\em groupoid} 
if
for each $x \in C$ there is an element $x^{-1}$ such that
$x^{-1}x = {\bf d}(x)$ and $xx^{-1} = {\bf r}(x)$.
The element $x^{-1}$ is unique with these properties. 
Two elements $x$ and $y$ of a groupoid are said to be 
{\em connected}
if there is an element starting at ${\bf d}(x)$
and ending at ${\bf d}(y)$.
This is an equivalence relation whose equivalence classes 
are called the {\em connected components} 
of the groupoid.
A groupoid with one connected component is said to be
{\em connected}.

Motivated by the symmetric inverse monoid, define the {\em restricted product}
in an inverse semigroup by $s \cdot t = st$ if $s^{-1}s = tt^{-1}$ and undefined otherwise.

\begin{proposition} Every inverse semigroup $S$ is 
a groupoid with respect to its restricted product.
\end{proposition}
\begin{proof}
We begin by showing that all idempotents
of $S$ are identities of $(S,\cdot)$.
Let $e \in S$ be an idempotent and suppose that $e \cdot x$ is defined.  
Then $e = xx^{-1}$ and $e \cdot x = ex$.  
But $ex = (xx^{-1})x = x$.  
Similarly, if $x \cdot e$ is defined then it is equal to $x$.  
We now check that the axioms (C1), (C2) and (C3) hold.

Axiom (C1) holds: suppose that $x \cdot (y \cdot z)$ is defined.  
Then 
$$x^{-1}x = (y \cdot z)(y \cdot z)^{-1} 
\mbox{ and } y^{-1}y = zz^{-1}.$$  
But 
$$(y \cdot z)(y \cdot z)^{-1} = yzz^{-1}y^{-1} = yy^{-1}.$$
Hence $x^{-1}x = yy^{-1}$, and so
$x \cdot y$ is defined.
Also $(xy)^{-1}(xy) = y^{-1}y = zz^{-1}$.
Thus $(x\cdot y) \cdot z$ is defined.
It is clear that $x \cdot (y \cdot z)$ is equal to
$(x \cdot y) \cdot z$. 
A similar argument shows that if $(x \cdot y) \cdot z$ exists 
then $x \cdot (y \cdot z)$ 
exists and they are equal. 

Axiom (C2) holds: suppose that $x \cdot y$ and $y \cdot z$ are defined.  
We show that $x \cdot (y \cdot z)$ is defined.  
We have that $x^{-1}x = yy^{-1}$ and $y^{-1}y = zz^{-1}$. 
Now 
$$(yz)(yz)^{-1} = y(zz^{-1})y^{-1} = y(y^{-1}y)y^{-1} = yy^{-1} = x^{-1}x.$$
Thus $x \cdot (y \cdot z)$ is defined.
The proof of the converse is straightforward.

Axiom (C3) holds: for each element $x$ we have that 
$x \cdot (x^{-1}x)$ is defined,
and we have seen that idempotents of $S$ are identities.  
Thus we put ${\bf d}(x) = x^{-1}x$.  
Similarly, we put $xx^{-1} = {\bf r}(x)$. 
It is now clear that $(S,\cdot)$ is a category.
The fact that it is a groupoid is immediate.
\end{proof}

We call $(S,\cdot)$ the {\em underlying groupoid} of $S$ or the {\em Ehresmann groupoid} of $S$
since it was first used by the differential geometer Charles Ehresmann.
The above result leads to the following pictorial representation of the elements of an inverse semigroup.
Recall that $\mathbf{d}(s) = s^{-1}s$, which we now call the {\em domain idempotent} of $s$, 
and that $\mathbf{r}(s) = ss^{-1}$, which we now call the {\em range idempotent} of $s$.
We can regard $s$ as an {\em arrow}
$$\diagram
&\mathbf{r}(s)
&
&\mathbf{d}(s) \llto_{s}
\enddiagram$$

The following result is more significant than it looks;
it will form the basis of Section~2 of Chapter~2.
If you draw a picture and imagine the elements are partial bijections you will see exactly what is going on.

\begin{proposition} Let $S$ be an inverse semigroup.
Then for any $s,t \in S$ there exist elements $s'$ and $t'$ such that
$st = s' \cdot t'$ where the product on the right is the restricted product.
\end{proposition}
\begin{proof} Put $e = \dom (s) \ran (t)$ and define $s' = se$ and $t' = et$.
Observe that $\dom (s') = e$ and $\ran (t') = e$ and that $st = s't'$.
\end{proof}

At this point, it is natural to define some relations, called {\em Green's relations}, 
which can be defined in any semigroup but assume particularly simple forms in inverse semigroups.
We define 
$s \mathcal{L} t$ iff $\dom (s) = \dom (t)$;
$s \mathcal{R} t$ iff $\ran (s) = \ran (t)$;
and $\mathcal{H} = \mathcal{L} \cap \mathcal{R}$ which corresponds to the $\mbox{hom}$-sets.
We define $s \mathcal{D} t$ iff $s$ and $t$ belong to the same connected component of the underlying groupoid.
If $\mathcal{K}$ is any one of Green's relation then $K_{s}$ denotes the $\mathcal{K}$-class containing $s$.

\begin{lemma} \mbox{} 
\begin{enumerate}

\item If $s \leq t$ and either $s \mathcal{L} t$ or $s \mathcal{R} t$ then $s = t$.

\item If $s \sim t$ and either $s \mathcal{L} t$ or $s \mathcal{R} t$ then $s = t$.

\item If $s \sim t$ and either $\dom(s) \leq \dom (t)$ or  $\ran (s) \leq \ran (t)$ then $s \leq t$.

\end{enumerate}
\end{lemma}
\begin{proof}
(1) Suppose that $s \leq t$ and $\dom (s) = \dom (t)$.
Then $s = ts^{-1}s = tt^{-1}t = t$.

(2) Suppose that $s \sim t$ and $\dom (s) = \dom (t)$.
Then $s \wedge t$ exists and $\dom (s \wedge t) = \dom (s)$ by Lemma~2.16.
By (1) above $s \wedge t = s$ and $s \wedge t = t$ and so $s = t$.

(3) Suppose that $s \sim t$ and $\dom(s) \leq \dom (t)$.
Then $s \wedge t$ exists and $\dom (s \wedge t) = \dom (s)$ by Lemma~2.16.
Thus $s \wedge t = s$ and so $s \leq t$.
 \end{proof}

If $\theta \colon S \rightarrow T$ then for each element $s \in S$ the map $\theta$ induces a function
from $L_{s}$ to $L_{\theta (s)}$ by restriction.
If all these restricted maps are injective (respectively, surjective)
we say that $\theta$ is {\em star injective} (respectively, {\em star surjective}).
In the literature, star injective homomorphisms are also referred to as {\em idempotent-pure}
maps on the strength of the following lemma.
We shall use this term when referring to congruences.

\begin{lemma} Let $\theta \colon S \rightarrow T$ be a homomorphism between inverse semigroups.
The following are equivalent
\begin{enumerate}

\item $\theta$ is  is star injective

\item Whenever $\theta (s)$ is an idempotent then $s$ is an idempotent.

\item The kernel of $\theta$ is contained in the compatibility relation.

\end{enumerate}

\end{lemma}
\begin{proof} (1)$\Rightarrow$(2).
Let $\theta$ be star injective and suppose that $\theta (s)$ is an idempotent.
Then $\theta (s^{-1}s) = \theta (s)$ since idempotents are self-inverse.
But $\theta$ is star injective and so $s^{-1}s = s$.

(2)$\Rightarrow$(3). Let $\theta (s) = \theta (t)$.
Then $\theta (s^{-1}s) = \theta (s^{-1}t)$ and so $s^{-1}t$ is an idempotent.
By symmetry $st^{-1}$ is an idempotent and so $s$ and $t$ are compatible.

(3)$\Rightarrow$(1). Let $\theta (s) = \theta (t)$ and $s \mathcal{L} t$.
Then $s \sim t$ and so $s = t$ by Lemma~2.24.
\end{proof}

The $E$-unitary inverse semigroups also arise naturally in the context of star injective homomorphisms.

\begin{theorem} Let $S$ be an inverse semigroup.  
Then the following conditions are equivalent: 
\begin{enumerate}

\item $S$ is $E$-unitary. 
 
\item $\sim ~ = \sigma $. 
 
\item $\sigma$ is idempotent pure.

\item $\sigma (e) = E(S)$ for any idempotent $e$.

\end{enumerate}
\end{theorem}
\begin{proof}
(1)$\Rightarrow$(2).  
We have already used the fact that the compatibility relation is contained in $\sigma$.
Let $(a,b) \in  \sigma $.  
Then $z \leq  a,b$ for some $z$.  
It follows that $z^{-1}z \leq  a^{-1}b$ and $zz^{-1} \leq  ab^{-1}$.  
But $S$ is $E$-unitary and so $a^{-1}b$ and $ab^{-1}$ are both idempotents.  
Hence $a \sim b$. 

(2)$\Rightarrow$(3). 
By Lemma~2.25 a congruence is idempotent pure precisely
when it is contained in the compatibility relation.

(3) $\Rightarrow$ (4).  
This is immediate from the definition of an idempotent pure congruence.

(4) $\Rightarrow (1)$
Suppose that $e \leq a$ where $e$ is an idempotent.
Then $(e,a) \in \sigma$.
But by (4), the element $a$ is an idempotent.
\end{proof}

The way in which the class of $E$-unitary inverse semigroups recurs is a reflection of the importance of
this class of inverse semigroups in the history of the subject.

In addition to the underlying groupoid, we may sometimes be able to associate another, smaller, groupoid to an inverse semigroup with zero.
Let $S$ be an inverse semigroup with zero.
An element $s \in S$ is said to be {\em $0$-minimal} if $t \leq s$ implies that $t = 0$ or $t = s$.
The set of $0$-minimal elements of $S$, if non-empty, forms a groupoid called the {\em minimal groupoid} of $S$.

\begin{example}
{\em The symmetric inverse monoid $I(X)$ has an interesting minimal groupoid.
It consists of those partial bijections who domains consist of exactly one element of $X$.
This groupoid is isomorphic to the groupoid $X \times X$ with product given by $(x,y)(y,z) = (x,z)$.
This is just the groupoid corresponding to the universal relation on $X$.
When $X$ is finite every partial bijection of $X$ can be written as an orthogonal join of elements of the minimal groupoid.
This simple example has far-reaching consequences as we shall see in a later chapter
}
\end{example}

\section{Some examples}\setcounter{theorem}{0}

So far, our range of examples of inverse semigroups is not very extensive.
This state of affairs is something we can now rectify using the tools we have available.
We describe three examples: groupoids with zero adjoined, presheaves of groups, and semidirect products of semilattices by groups.

\subsection{Groupoids with zero adjoined}

Category theorists may shudder at this example but a similar idea lies behind the construction of matrix rings from matrix units.

\begin{proposition}
Groupoids with zero adjoined are precisely the inverse semigroups in which the natural partial order is equality
when restricted to the set of non-zero elements.
\end{proposition}
\begin{proof} If $G$ is a groupoid then $S = G^{0}$, the groupoid $G$ with an adjoined zero, 
is a semigroup when we define all undefined product to be zero. 
It is an inverse semigroup and the natural partial order is equality when restricted to the non-zero elements.

To prove the converse, let $S$ be an inverse semigroup in which the natural partial order is equality when restricted to the set of non-zero elements.
Let $s$ and $t$ be arbitrary elements in $S$.
If $\dom (s) = \ran (t)$ then $st$ is just the restricted product.
Suppose that $\dom (s) \neq \ran (t)$. 
Then  $\dom (s) \ran (t) = 0$. 
It follows that in this case $st = 0$.
Thus the only non-zero products in $S$ are the restricted products
and the result follows.
\end{proof}

\subsection{Presheaves of groups}

The idempotents of an inverse semigroup commute amongst themselves but needn't commute with anything else.
The extreme case where they do is interesting.
An inverse semigroup is said to be {\em Clifford} if its idempotents are central.
Abelian inverse semigroups are Clifford semigroups and play a central role in the cohomology of inverse semigroups.
We show first how to construct examples of Clifford semigroups.

Let $(E,\leq )$ be a meet semilattice, and 
let $\{G_{e} \colon \: e \in E\}$ be a family of disjoint groups indexed 
by the elements of $E$,
the identity of $G_{e}$ being denoted by $1_{e}$.
For each pair $e,f$ of elements of $E$ where $e \geq f$ let
$\phi_{e,f} \colon \: G_{e} \rightarrow G_{f}$ be a group homomorphism,
such that the following two axioms hold:

\begin{description}

\item[{\rm (PG1)}] $\phi_{e,e}$ is the identity homomorphism on $G_{e}$.

\item[{\rm (PG2)}] If $e \geq f \geq g$ then 
$ \phi_{f,g} \phi_{e,f} = \phi_{e,g}$.

\end{description}
We call such a family  
$$(G_{e},\phi_{e,f}) 
= (\{G_{e} \colon \: e \in E \}, \{\phi_{e,f} \colon \: e,f \in E, f \leq e \})$$ 
a {\em presheaf of groups (over the semilattice $E$)}.

\begin{proposition} Let $(G_{e},\phi_{e,f})$ be a presheaf of groups.
Let $S = S(G_{e},\phi_{e,f})$ be the union of the $G_{e}$  
equipped with the product defined by:  
$$xy = \phi_{e,e \wedge f}(x)\phi_{f,e \wedge f}(y),$$             
where $x \in  G_{e}$ and $y \in  G_{f}$. 
With respect to this product, $S$ is a Clifford semigroup.
\end{proposition}
\begin{proof}
The product is clearly well-defined.  
To prove associativity, let $x \in G_{e}$, $y \in G_{f}$ and $z \in G_{g}$
and put $i = e \wedge f \wedge g$.
By definition
$$(xy)z 
= \phi_{e \wedge f,i}(\phi_{e,e \wedge f}(x) \phi_{f,e \wedge f}(y)) \phi_{g,i}(z).$$
But
$$\phi_{e \wedge f,i} (\phi_{e,e \wedge f}(x) \phi_{f,e \wedge f}(y)) =      
\phi_{e \wedge f,i}(\phi_{e,e \wedge f}(x)) 
\phi_{e \wedge f,i}(\phi_{f,e \wedge f}(y)).$$
By axiom (PG2) this simplifies to
$\phi_{e,i}(x)\phi_{f,i}(y)$.
Thus 
$$(xy)z = \phi_{e,i}(x) \phi_{f,i}(y) \phi_{g,i}(z).$$
A similar argument shows that $x(yz)$
likewise reduces to the right-hand side of the above equation.
Thus $S$ is a semigroup.

Observe that if $x,y \in G_{e}$ then $xy$ is just their product in $G_{e}$.
Thus if $x \in G_{e}$ and $x^{-1}$ is the inverse 
of $x$ in the group $G_{e}$ then  
$$x = xx^{-1}x \mbox{ and }
x^{-1} = x^{-1}xx^{-1}$$
by axiom (PG1).
Thus $S$ is a regular semigroup.

The idempotents of $S$ are just the identities of the groups $G_{e}$,  
again by axiom (PG1)
and $1_e 1_f = 1_{e \wedge f}$.  
Thus the idempotents commute.  
We have thus shown that $S$ is an inverse semigroup.

To finish off, let $x \in G_{f}$. 
Then
$$1_e x 
= \varphi_{e,e \wedge f} (1_e) \varphi_{f,e \wedge f} (x) 
= 1_{e \wedge f} \varphi_{f,e \wedge f} (x) 
= \varphi_{f,e \wedge f} (x),$$
and similarly, $x 1_e = \varphi_{f,e \wedge f} (x)$.  
Consequently, the idempotents of $S$ are central. 
\end{proof}

The underlying groupoid of a Clifford semigroup is just a union of groups as the following lemma shows.

\begin{lemma} Let $S$ be an inverse semigroup.
Then $S$ is Clifford if and only if $s^{-1}s = ss^{-1}$
for every $s \in S$.
\end{lemma}
\begin{proof}
Let $S$ be a Clifford semigroup and let $s \in S$.
Since the idempotents are central
$s = s(s^{-1}s) = (s^{-1}s)s$.
Thus $ss^{-1} \leq s^{-1}s$.
We may similarly show that $s^{-1}s \leq ss^{-1}$, 
from which we obtain $s^{-1}s = ss^{-1}$.

Suppose now that $s^{-1}s = ss^{-1}$ for all elements $s$.
Let $e$ be any idempotent and $s$ an arbitrary element.
Then $(es)^{-1}es = es(es)^{-1}$.
That is $s^{-1}es = ss^{-1}e$.
Multiplying on the left by $s$ gives $es = se$, as required.
\end{proof}

We may now characterize Clifford inverse semigroups.

\begin{theorem} 
An inverse semigroup is a Clifford semigroup if and only if it is isomorphic to a presheaf of groups.
\end{theorem}
\begin{proof} Let $S$ be a Clifford semigroup.
By Lemma~3.3, we know that $s^{-1}s = ss^{-1}$ for all element $s$.
This implies that the underlying groupoid of $S$ is a union of groups.
For each idempotent $e \in E(S)$ define
$$G_{e} = \{s \in S \colon \dom (s) = e = \ran (s)    \}.$$
This is a group, the local group at the identity $e$ in the underlying groupoid.
By assumption the union of these groups is the whole of $S$ and each element of $S$ belongs to exactly one of these groups.
If $e \geq f$ define $\phi_{e,f} \colon \: G_{e} \rightarrow G_{f}$
by $\phi_{e,f}(a) = af$.
This is a well-defined function, because $\dom (af) = e$.
We show that $(G_{e},\phi_{e,f})$ is a presheaf of
groups over the semilattice $E(S)$.

Axiom (PG1) holds: let $e \in E(S)$ and $a \in G_{e}$.
Then $\phi_{e,e}(a) = ae = aa^{-1}a = a$. 

Axiom (PG2) holds: let $e \geq f \geq g$ and $a \in G_{e}$.
Then 
$$(\phi_{f,g}\phi_{e,f})(a) = \phi_{f,g}(\phi_{e,f}(a)) 
= afg = ag = \phi_{e,g}(a).$$

Let $T$ be the inverse semigroup constructed from this presheaf of groups. 
Let $a \in G_{e}$ and $b \in G_{f}$.
We calculate their product in this semigroup.
By definition
$$\phi_{e,ef}(a)\phi_{f,ef}(b) = aefbef = afbe = aefb = ab.$$
Thus $S$ and $T$ are isomorphic.

The converse was proved in Proposition~3.2.
\end{proof}

\subsection{Semidirect products of  semilattices by groups}

The group $G$ acts on the set $Y$ (on the left) if there is a function 
$G \times Y \rightarrow Y$ denoted by $(g,e) \mapsto g \cdot e$
satisfying $1 \cdot e = e$ for all $e \in Y$  and
$g \cdot (h \cdot e) = (gh) \cdot e$ for all $g,h \in G$ and $e \in Y$.
If $Y$ is a partially ordered set, 
then we say that $G$ acts on $Y$
by {\em order automorphisms} if for all $e,f \in Y$
we have that 
$$e \leq f \Leftrightarrow g \cdot e \leq g \cdot f.$$
Observe that in the case of a group action, 
it is enough to assume that
$e \leq f$ implies $g \cdot e \leq g \cdot f$, 
because if $g \cdot e \leq g \cdot f$ 
then $g^{-1} \cdot (g \cdot e) \leq g^{-1} \cdot (g \cdot f)$
and so $1 \cdot e \leq 1 \cdot f$, which gives $e \leq f$.
If $Y$ is a meet semilattice on which $G$ acts by order automorphisms,
then it is automatic that 
$$g \cdot (e \wedge f) = g \cdot e \wedge g \cdot f$$
for all $g \in G$ and $e,f \in Y$.   

Let $P(G,Y)$ be the set $Y \times  G$ equipped with the multiplication
$$(e,g)(f,h) = (e \wedge  g \cdot f,gh).$$

\begin{proposition} $P(G,Y)$ is an $E$-unitary inverse semigroup 
in which the semilattice of idempotents is isomorphic to $(Y,\leq)$
and $G$ is isomorphic to the maximum group homomorphic image of $P(G,Y)$.
\end{proposition}
\begin{proof} $P(G,Y)$ is an inverse semigroup in which 
the inverse of $(e,g)$ is the element $(g^{-1} \cdot e,g^{-1})$, and 
the idempotents of $P(G,Y)$ are the elements of the form $(e,1)$. 
From the definition of the multiplication in $P(G,Y)$
the function $(e,1) \mapsto e$ is an isomorphism of semilattices.
The natural partial order is given by
$$(e,g) \leq (f,h) \Leftrightarrow e \leq f \mbox{ and } g = h.$$
If $(e,1) \leq (f,g)$ then $g = 1$ and so
$P(G,Y)$ is $E$-unitary.
It also follows from the description of the natural partial order that
$(e,g) \sigma (f,h)$ if and only if $g = h$.
\end{proof}

We may now characterize those inverse semigroups isomorphic to semidirect products of semilattices by groups
using many of the ideas introduced in Section~2 to do so.

\begin{theorem} Let $S$ be an inverse semigroup.  
Then the following are equivalent:  

\begin{enumerate}

\item  The semigroup $S$ is isomorphic to a 
semidirect product of a semilattice by a group.

\item $S$ is $E$-unitary and for each $a \in S$ and $e \in E(S)$ 
there exists $b \in S$ such that $b \sim a$ and $b^{-1}b = e$. 

\item $\sigma^{\natural} \colon \: S \rightarrow S/\sigma$ is star bijective. 

\item There is a star bijective homomorphism from $S$ to a group.

\item The function 
$\theta \colon \: S \rightarrow E(S) \times S/\sigma$ 
defined by $\theta (a) = (a^{-1} a, \sigma (a))$ is a bijection.  

\item The function 
$\phi \colon \: S \rightarrow E(S) \times S/\sigma$
defined by $\phi (a) = (aa^{-1}, \sigma (a))$ is a bijection.

\end{enumerate}
\end{theorem}
\begin{proof}
(1) $\Rightarrow $ (2).  
Without loss of generality, we may assume that
$S$ is a semidirect product of a meet semilattice $Y$ by a group $G$.  
The semigroup $S$ is $E$-unitary by Theorem~3.6.  
Let $(e,g) \in S$ and $(f,1) \in E(S)$. 
Then the element $(g \cdot f,g)$ of $S$ satisfies
$$(g \cdot f,g) \sim (e,g)
\mbox{ and }
(g \cdot f,g)^{-1}(g \cdot f,g) = (f,1)$$ 
as required.

(2) $\Rightarrow $ (3).  
Since $S$ is $E$-unitary, the homomorphism
$\sigma^{\natural} \colon \: S \rightarrow S/\sigma$ 
is star injective by Theorem~2.26.
Let $e \in E(S)$ and $\sigma (a) \in S/\sigma$.    
By assumption there exists $b \in S$ such that 
$b^{-1}b = e$ and $b \sim a$.  
But $b \sim a$ implies $\sigma (b) = \sigma (a)$.  
Thus $\sigma^{\natural}$ is also star surjective. 

(3) $\Rightarrow $ (4). Immediate. 

(4) $\Rightarrow $ (3).  
Let $\theta \colon \: S\rightarrow G$ 
be a star bijective homomorphism to a group $G$.  
Since $\sigma$ is the minimum group congruence,
$\sigma \subseteq \mbox{ker}\,\theta$ by Theorem~2.19. 
But $\theta$ is star injective by assumption,
and so $\sigma^{\natural}$ is idempotent pure by Lemma~2.25.  
In particular, $S$ is $E$-unitary by Theorem~2.26.

To show that $\sigma^{\natural}$ is star surjective, 
let $s \in S$ and $e \in E(S)$.  
There exists $t \in S$ such that $t^{-1}t  = e$ and $\theta (t) = \theta(s)$,
since $\theta$ is star surjective.  
Now $\theta (s^{-1}t)$ is the identity of $G$, 
and so $s^{-1}t$ is an idempotent of $S$  
since $\theta$ is star injective.
Similarly, $st^{-1}$ is an idempotent.  
Hence $s \sim t$ and so $(s,t) \in \sigma$. 
Thus for each $e \in E(S)$ and $\sigma (s) \in S/\sigma$,
there exists $t \in S$ such that $t^{-1}t = e$ and $\sigma (t) = \sigma (s)$.
Thus $\sigma^{\natural}$ is star surjective.

(3) $\Rightarrow$ (5). Straightforward. 

(5) $\Rightarrow$ (6).  
Suppose that $\phi (a) = \phi(b)$.  
Then $aa^{-1} = bb^{-1}$ and $\sigma (a) = \sigma (b)$.  
But $\sigma (a^{-1}) = \sigma (b^{-1})$
and so $\theta (a^{-1}) = \theta (b^{-1})$.  
By assumption $\theta$ is bijective and so
$a^{-1} = b^{-1}$, giving $a = b$.  
Hence $\phi$ is injective.

Now let $(e, \sigma (s)) \in E \times S/\sigma$.  
Since $\theta$ is surjective there exists $t \in S$ such that
$\theta (t) = (e, \sigma (s^{-1}))$.  
Thus $t^{-1}t = e$ and $t \, \sigma \, s^{-1}$.  
Hence $t^{-1}$ is such that $t^{-1}\, \sigma \,s$ 
and $t^{-1}(t^{-1})^{-1} = e$.  
Thus $\phi (t^{-1}) = (e, \sigma (s))$, 
and so $\phi$ is surjective.

(6) $\Rightarrow$ (5).  
A similar argument to (5) $\Rightarrow$ (6).

(6) $\Rightarrow $ (1).  
We shall use the fact that both the
functions $\phi$ and $\theta$ defined above are bijections.  

First of all $S$ is $E$-unitary. 
For suppose that $e \leq a$ where $e$ is an idempotent.  
Then $\sigma (e) = \sigma (a)$, 
and $\sigma (e) = \theta (a^{-1}a)$,
so that $\sigma (a) = \sigma (a^{-1}a)$.
Thus $\theta (a) = \theta (a^{-1}a)$,
and so $a = a^{-1}a$, since $\theta$ is a bijection.  

We shall define an action of $S/\sigma$ on $E(S)$ 
using $\theta$, and then show that
$\phi$ defines an isomorphism from the
semidirect product of $E(S)$ by $S/\sigma$ to $S$.  

Define $\sigma (s) \cdot e = tt^{-1}$ where $\theta (t) = (e, \sigma (s))$.  
This is well-defined because $\theta$ is a bijection.
The two defining properties of an action hold.
Firstly, if $\sigma (e)$ is the identity of $S/\sigma$ 
then $\theta (e) = (e, \sigma (e))$ and so $\sigma (e) \cdot e = e$; 
secondly, $\sigma (u) \cdot (\sigma (v) \cdot e) = \sigma (u) \cdot aa^{-1}$ 
where $\theta (a) = (e,\sigma (v))$, 
and $\sigma (u) \cdot aa^{-1} = bb^{-1}$ where
$\theta (b) = (aa^{-1}, \sigma (u))$.  
Now $a \, \sigma \, v$ and $b \, \sigma \, u$ 
so that $ba \, \sigma \, uv$.
Also $a^{-1}a = e$ and $b^{-1}b = aa^{-1}$ so that
$(ba)^{-1}ba = a^{-1}a$.
Hence
$\theta (ba) = (e, \sigma (uv))$.
Thus
$$\sigma (uv) \cdot e = (ba)(ba)^{-1} = bb^{-1}
= \sigma (u) \cdot (\sigma (v) \cdot e).$$
 
Next, we show that $S/\sigma$ acts on $E(S)$ by means of order automorphisms.  
Suppose that $e \leq f$.
Then $\sigma (a) \cdot e = uu^{-1}$ and $\sigma (a) \cdot f = vv^{-1}$ where
$$\theta (u) = (e,\sigma(a)) \mbox{ and } \theta (v) = (f,\sigma (a)).$$
Consequently, $e = u^{-1}u$ and $f = v^{-1}v$
and $u\, \sigma \,v$.
But $S$ is $E$-unitary, and so
$\sigma$ is equal to the compatibility relation 
by Theorem~2.26.
From $u^{-1}u \leq v^{-1}v$ and $u \sim v$
we obtain $u \leq v$ by Lemma~2.24.
Hence $uu^{-1} \leq vv^{-1}$ and so
$\sigma (a) \cdot e \leq \sigma (a) \cdot f$.  

It only remains to prove that $\phi$ is a homomorphism.  
By definition
$$\phi(a) \phi(b) = (aa^{-1},\sigma(a))(bb^{-1},\sigma(b))
= (aa^{-1} \wedge \sigma (a) \cdot bb^{-1},\sigma(ab)).$$
But 
$\sigma (a) \cdot bb^{-1} = tt^{-1}$ 
where $\theta (t) = (bb^{-1},\sigma (a))$.  
Thus
$$\phi (a)\phi(b) = (aa^{-1}tt^{-1}, \sigma (ab))$$
whereas 
$$\phi(ab) = (ab(ab)^{-1},\sigma(ab)).$$
It remains to show that
$aa^{-1}tt^{-1} = ab(ab)^{-1}$.
We know that $t^{-1}t = bb^{-1}$ and $t\, \sigma \, a$. 
But $t \sim a$ since $S$ is $E$-unitary. 
Thus $tt^{-1}a = at^{-1}t = abb^{-1}$ by Lemma~2.16.  
Hence $tt^{-1}aa^{-1} = abb^{-1}a^{-1} = ab(ab)^{-1}$.
\end{proof}

\section{Fundamental inverse semigroups}\setcounter{theorem}{0}

The examples in the last section can be viewed as showing that various natural ways of combining groups and semilattices lead  to interesting classes of inverse semigroups.
But what does the `generic' inverse semigroup look like?
The main goal of this section is to justify the claim made in the Introduction that inverse semigroups should be viewed as 
common generalizations of presheaves of groups and pseudogroups of transformations.
We shall also characterize the congruence-free inverse semigroups with zero.
Interesting examples of such semigroups will be discussed later.

\subsection{The Munn representation}

The symmetric inverse monoid is constructed from an arbitrary set.
We now show how to construct an inverse semigroup from a meet semilattice.
Let $(E,\leq )$ be a meet semilattice, 
and denote by $T_{E}$ be the set of all order isomorphisms between principal order ideals of $E$.  
Clearly, $T_{E}$ is a subset of $I(E)$.
In fact we have the following.

\begin{proposition} The set $T_{E}$ is an inverse subsemigroup of $I(E)$
whose semilattice of idempotents is isomorphic to E.
\end{proposition}

$T_{E}$ is called the {\em Munn semigroup} of the semilattice $E$.

\begin{theorem}[Munn representation theorem]  
Let $S$ be an inverse semigroup.  
Then there is an idempotent-separating homomorphism 
$\delta \colon S \rightarrow T_{E(S)}$ 
whose image is a wide inverse subsemigroup of $T_{E(S)}$.
\end{theorem}
\begin{proof}
For each $s \in S$  
define the function 
$$\delta _{s} \colon \: (s^{-1}s)^{\downarrow} \rightarrow  (ss^{-1})^{\downarrow}$$ 
by $\delta_{s}(e) = ses^{-1}$.
We first show that $\delta_{s}$ is well-defined.  
Let $e \leq  s^{-1}s$.  
Then $ss^{-1}\delta_{s}(e) = \delta_{s}(e)$, 
and so $\delta_{s}(e) \leq  ss^{-1}$.  
To show that $\delta_{s}$ is order-preserving, let $e \leq f \in (s^{-1}s)^{\downarrow}$.  
Then
$$\delta_{s}(e)\delta_{s}(f) = ses^{-1}sfs^{-1} 
= sefs^{-1} = \delta_{s}(e).$$
Thus $\delta_{s}(e) \leq \delta_{s}(f)$.

Consider now the function
$\delta_{s^{-1}} \colon  (ss^{-1})^{\downarrow} \to (s^{-1}s)^{\downarrow}$.
This is order-preserving by the argument above.  
For each $e \in  (s^{-1}s)^{\downarrow}$, we have that 
$$\delta_{s^{-1}}(\delta_{s}(e)) = \delta_{s^{-1}}(ses^{-1}) 
= s^{-1}ses^{-1}s = e.$$
Similarly, $\delta_{s}(\delta_{s^{-1}}(f)) = f$ for each $f \in  (ss^{-1})^{\downarrow}$. 
Thus $\delta_{s}$ and $\delta_{s^{-1}}$ are mutually inverse, 
and so $\delta_{s}$ is an order isomorphism.

Define $\delta \colon \: S \rightarrow T_{E(S)}$
by $\delta (s) = \delta_{s}$.
To show that $\delta$ is a homomorphism,  
we begin by calculating $\mbox{dom}(\delta_{s} \delta_{t})$
for any $s,t \in S$.
We have that
$$\mbox{dom} (\delta_{s} \delta_{t})  
=  \delta^{-1}_{t} ((s^{-1}s)^{\downarrow} ~ \cap ~ (tt^{-1})^{\downarrow}) 
= \delta^{-1}_{t} ((s^{-1}stt^{-1})^{\downarrow}).$$
But $\delta^{-1}_{t} = \delta_{t^{-1}}$ and so 
$$\mbox{dom} (\delta_{s}  \delta_{t}) 
= ((st)^{-1}st)^{\downarrow} = \mbox{dom}(\delta_{st}).$$
If $e \in \mbox{dom}\,\delta_{st}$ then 
$$\delta_{st}(e) = (st)e(st)^{-1} = s(tet^{-1})s^{-1}
= \delta_{s}(\delta_{t}(e)).$$  
Hence $\delta_{s} \delta_{t} = \delta_{st}$.  

To show that $\delta$ is idempotent-separating, suppose that 
$\delta (e) = \delta (f)$ where $e$ and $f$ are idempotents of $S$.  
Then $\mbox{dom}\,\delta (e) = \mbox{dom}\,\delta (f)$.  
Thus $e = f$.  

The image of $\delta$ is a wide inverse
subsemigroup of $T_{E(S)}$
because every idempotent in $T_{E(S)}$ is of the form
$1_{[e]}$ for some $e \in E(S)$,
and $\delta_{e} = 1_{[e]}$.\end{proof}

The Munn representation should be contrasted with the Wagner-Preston representation:
that was injective whereas this has a non-trivial kernel which we shall now describe.
The kernel of $\delta$ is the congruence $\mu$ defined by $(s,t) \in \mu$ 
if and only if $\dom (s) = \dom (t)$, $\ran (s) = \ran (t)$ and for all idempotents
$e$ such that $e \leq s^{-1}s$ we have that $ses^{-1} = tet^{-1}$.
The definition can be slightly weakened.

\begin{lemma} The congruence $\mu$ is defined by
$$(s,t) \in \mu \Leftrightarrow (\forall e \in E(S)) \,ses^{-1} = tet^{-1}.$$
\end{lemma}
\begin{proof} Define $(s,t) \in \mu'$ iff $ses^{-1} = tet^{-1}$ for all idempotents $e$.
We shall prove that $\mu = \mu'$.
Observe first that $\mu'$ is a congruence.
It is clearly an equivalence relation.
Suppose that $(a,b) \in \mu'$ and $(c,d) \in \mu'$.
The proof that $(ac,bd) \in \mu'$ is straightforward.
It follows that from $(s,t) \in \mu'$ we may deduce that $(s^{-1},t^{-1}) \in \mu'$.
Let $(s,t) \in \mu'$.
We prove that $(s,t) \in \mu$.
To do this we need to prove that $\dom (s) = \dom (t)$, $\ran (s) = \ran (t)$.
By choosing our idempotent to be $ss^{-1}$ we get that $ss^{-1} \leq tt^{-1}$.
By symmetry we deduce that $\ran (s) = \ran (t)$.
The fact that $\dom (s) = \dom (t)$ follows from the same argument using the fact that $(s^{-1},t^{-1}) \in \mu'$. 
We have shown that $\mu' \subseteq \mu$.

To prove the converse, suppose that $(s,t) \in \mu$.
Let $e$ be an arbitrary idempotent.
Then $s^{-1}s = t^{-1}t$ and so $s^{-1}se = t^{-1}te$.
Thus $s(s^{-1}se)s^{-1} = t(t^{-1}te)t^{-1}$,
which simplifies to $ses^{-1} = tet^{-1}$.
It follows that $(s,t) \in \mu'$, as required.
\end{proof}

We have defined idempotent-separating homomorphisms and we may likewise define idempotent-separating congruences.

\begin{lemma} 
$\mu$ is the largest idempotent-separating congruence on $S$.
\end{lemma}
\begin{proof}
Let $\rho$ be any idempotent separating-congruence on $S$ and let $(s,t) \in \rho$.
Let $e$ be any idempotent.
Then $(ses^{-1},tet^{-1}) \in \rho$
but $\rho$ is idempotent separating and so $ses^{-1} = tet^{-1}$.
It follows that $(s,t) \in \mu$.
Thus we have shown that $\rho \subseteq \mu$. 
\end{proof}

An inverse semigroup is said to be {\em fundamental} if $\mu$ is the equality relation.

\begin{lemma} Let $S$ be an inverse semigroup.
Then $S/\mu$ is fundamental.
\end{lemma}
\begin{proof}
Suppose that $\mu (s)$ and $\mu (t)$ are $\mu$-related in $S/\mu$.
Every idempotent in $S/\mu$ is of the form $\mu (e)$ where $e \in E(S)$.
Thus
$$\mu (s)\mu (e)\mu (s)^{-1} = \mu (t)\mu (e) \mu (t)^{-1}$$
so that $\mu (ses^{-1}) = \mu (tet^{-1})$.
But both $ses^{-1}$ and $tet^{-1}$ are idempotents, so that
$ses^{-1} = tet^{-1}$ for every $e \in E(S)$.
Thus $(s,t) \in \mu$.
\end{proof}

\begin{theorem}  Let $S$ be an inverse semigroup.  
Then $S$ is fundamental if, and only if, 
$S$ is isomorphic to a wide inverse subsemigroup 
of the Munn semigroup $T_{E(S)}$.
\end{theorem}
\begin{proof}
Let $S$ be a fundamental inverse semigroup. 
By Theorem~4.2, 
there is a homomorphism $\delta \colon \: S \rightarrow T_{E(S)}$
such that $\mbox{ker}\,\delta = \mu$.
By assumption, $\mu$ is the equality congruence, and so $\delta$
is an injective homomorphism.
Thus $S$ is isomorphic to its image in $T_{E(S)}$,
which is a wide inverse subsemigroup.

Conversely, let $S$ be a wide inverse subsemigroup of a Munn
semigroup $T_{E}$.
Clearly, we can assume that $E = E(S)$.
We calculate the maximum idempotent-separating congruence of $S$. 
Let $\alpha,\beta \in S$ and suppose that $(\alpha,\beta) \in \mu$ in $S$.
Then $\mbox{dom}\,\alpha = \mbox{dom}\,\beta$.
Let $e \in \mbox{dom}\,\alpha$.
Then $1_{[e]} \in S$, since $S$ is a wide inverse subsemigroup of $T_{E(S)}$.
By assumption
$\alpha 1_{[e]} \alpha^{-1} = \beta 1_{[e]} \beta^{-1}$.
It is easy to check that 
$1_{[\alpha (e)]} = \alpha 1_{[e]} \alpha^{-1}$
and
$1_{[\beta (e)]} = \beta 1_{[e]}  \beta^{-1}$.
Thus $\alpha (e) = \beta (e)$. 
Hence $\alpha = \beta$, and so $S$ is fundamental.
\end{proof}

In group theory, congruences are handled using normal subgroups, and in ring theory by ideals.
In general semigroup theory, there are no such substructures and so congruences have to be studied in their own right 
something that is common to most of universal algebra.
Even in the case of inverse semigroups, congruences have to be used.
However, idempotent-separating homomorphisms are determined by analogues of normal subgroups. 

Let $\theta \colon S \rightarrow T$ be a homomorphism of inverse semigroups.
The {\em Kernel of $\theta$} is defined to be the set $K$ of all elements of $S$ that map to idempotents under $\theta$.
Observe that $K$ is a wide inverse subsemigroup of $S$ and it is {\em self-conjugate}
in the sense that $s^{-1}Ks \subseteq K$ for all $s \in S$.
We say that $K$ is a a {\em normal inverse subsemigroup} of $S$.

\begin{remark}{\em  
This typographical distinction between {\em kernels} which are congruences and {\em Kernels}
which are substructures is not entirely happy but convenient for the purposes of this section.}
\end{remark}

If $\theta$ is idempotent-separating then its Kernel satisfies an additional property.
If $a \in K$ and if $e$ is any idempotent then $ae = ea$.
This motivates the following definition.

For every inverse semigroup $S$, we define $Z(E(S))$, the {\em centralizer of the idempotents},
to be set of all elements of $S$ which commute with every idempotent.
The centralizer is a normal inverse subsemigroup and is Clifford.
Thus the Kernels of idempotent-separating homomorphisms from $S$ are subsets of the centralizer of the idempotents of $S$.
We now prove that idempotent-separating homomorphisms are determined by their Kernels.

\begin{theorem} Let $S$ be an inverse semigroup.
Let $K$ be a normal inverse subsemigroup of $S$ contained in $Z(E(S))$.
Define the relation $\rho_{K}$ by 
$$(s,t) \in \rho_{K} \Leftrightarrow  st^{-1} \in K \text{ and }\dom (s) = \dom (t).$$
Then $\rho_{K}$ is an idempotent-separating congruence whose associated Kernel is $K$.
\end{theorem}
\begin{proof}
We show first that $\rho_{K}$ is an equivalence relation.
Reflexivity and symmetry hold because $K$ is a wide inverse subsemigroup of $S$.
To prove transitivity suppose that $(a,b),(b,c) \in \rho_{K}$.
Then 
$ab^{-1},bc^{-1} \in K$
and $\dom (a) = \dom (b) = \dom (c)$.
Observe that
$ab^{-1}bc^{-1} = ac^{-1} \in K$ and $\dom (a) = \dom (c)$.
Hence $(a,c) \in \rho_{K}$.
Next we show that $\rho_{K}$ is a congruence.
Let $(a,b) \in \rho_{K}$ and $c \in S$.
By assumption, $ab^{-1} \in K$ and $\dom (a) = \dom (b)$.
We prove first that $\rho_{K}$ is a right congruence by showing that $(ac,bc) \in \rho_{K}$.
Observe that 
$ac(bc)^{-1} = acc^{-1}b^{-1}$.
We may move the idempotent $cc^{-1}$ through $b^{-1}$ by Lemma~2.6.
Thus by the fact that $K$ is a wide inverse subsemigroup
we have show that $ac(bc)^{-1} \in K$.
A simple calculation shows that $\dom (ac) = \dom (bc)$.
We prove now that $\rho_{K}$ is a left congruence by showing that $(ca,cb) \in \rho_{K}$.
Observe that $ca(cb)^{-1} = c(ab^{-1})c^{-1}$, but $ab^{-1} \in K$ and $K$ is self-conjugate so that $ca(cb)^{-1} \in K$.

It remains to show that the elements
$$(ca)^{-1}ca = a^{-1}c^{-1}ca
\mbox{ and }
(cb)^{-1}cb = b^{-1}c^{-1}cb$$
are equal.
Put $e = c^{-1}c$.
We shall show that $a^{-1}ea = b^{-1}eb$.
Write
$$a^{-1}ea = (a^{-1}ea)(a^{-1}a)(a^{-1}ea).$$
But $a^{-1}a = b^{-1}b$
and so
$$a^{-1}ea = (a^{-1}ea)(b^{-1}b)(a^{-1}ea).$$
Now
$$(a^{-1}ea)(b^{-1}b)(a^{-1}ea) = (a^{-1}e)(ab^{-1})(ab^{-1})^{-1}(ea).$$
But $ab^{-1} \in K$, and $K$ is contained in the centralizer of the idempotents, and so
$$ab^{-1}(ab^{-1})^{-1} = (ab^{-1})^{-1}ab^{-1}.$$
Thus 
$$(a^{-1}e)(ab^{-1})(ab^{-1})^{-1}(ea) 
= 
(a^{-1}e)(ab^{-1})^{-1}(ab^{-1})(ea),$$                
and so
$$a^{-1}ea = (a^{-1}e)(ba^{-1}ab^{-1})(ea).$$
Now
$$(a^{-1}e)(ba^{-1}ab^{-1})(ea) = a^{-1}(ab^{-1}e)^{-1}(ab^{-1}e)a.$$
But $ab^{-1} \in K$, and $K$ is a wide subsemigroup, so that $ab^{-1}e \in K$.
Thus because $K$ is contained in the centralizer of the idempotents we have that
$$a^{-1}(ab^{-1}e)^{-1}(ab^{-1}e)a 
= 
a^{-1}(ab^{-1}e)(ab^{-1}e)^{-1}a.$$
Thus
$$a^{-1}ea = a^{-1}(ab^{-1}e)(ab^{-1}e)^{-1}a.$$
But $a^{-1}(ab^{-1}e)(ab^{-1}e)^{-1}a = a^{-1}ab^{-1}eb$,
so that we in fact have
$$a^{-1}ea = a^{-1}ab^{-1}eb.$$
But then from $a^{-1}a = b^{-1}b$
we obtain 
$a^{-1}ea =  b^{-1}eb$  as required. 

We now calculate the Kernel of $\rho_{K}$.
Let $a$ be in the Kernel of $\rho_{K}$. 
Then there is an idempotent $e \in S$ such that
$(a,e) \in \rho_{K}$.
But then $ae \in K$ and $a^{-1}a = e$.
Thus $a \in K$.
It follows that the Kernel of  $\rho_{K}$ is contained in $K$.
To prove the reverse inclusion, suppose that $a \in K$.
Then $a(a^{-1}a) \in K$ and $a^{-1}a = a^{-1}a$.
Thus $(a,a^{-1}a) \in \rho_{K}$.
Hence $a$ belongs to the Kernel of $\rho_{K}$.
\end{proof}

The following now confirms what we already suspect.

\begin{proposition} 
Let $S$ be an inverse semigroup.
The idempotent-separating congruence determined by $Z(E(S))$ is $\mu$.
\end{proposition}
\begin{proof} We calculate the Kernel of $\mu$.
Suppose that $s \mu e$ where $e$ is an idempotent.
Let $f$ be an arbitrary idempotent.
Then $sfs^{-1} \mu ef$ and $fss^{-1} \mu ef$.
Thus $sfs^{-1} \mu fss^{-1}$ and so $sfs^{-1} = fss^{-1}$.
It follows that $sf = fs$ and $s \in Z(E(S))$.
Conversely, let $s \in Z(E(S))$.
Then $s \mu ss^{-1}$. 
\end{proof}

The following result provides a useful criterion for a semigroup to be fundamental.

\begin{proposition} Let $S$ be an inverse semigroup.
Then $S$ is fundamental if, and only if, $Z(E(S)) = E(S)$.
\end{proposition}
\begin{proof}
Suppose that $S$ is fundamental.
Let $a \in Z(E(S))$.
By Proposition~4.9, $\mbox{Ker}\,\mu = Z(E(S))$.
Thus $(a,e) \in \mu$ for some $e \in E(S)$.
But then $a = e$, since $\mu$ is equality,
and so $a$ is an idempotent.
Thus $Z(E(S)) = E(S)$.

Conversely, suppose that $Z(E(S)) = E(S)$.
Let $(a,b) \in \mu$.
Then 
$$(ab^{-1},bb^{-1}) \in \mu,$$ 
and so $ab^{-1} \in \mbox{Ker}\, \mu$. 
But $\mbox{Ker}\, \mu = Z(E(S))$ by Proposition~4.9, 
and so $ab^{-1} \in Z(E(S))$. 
Thus $ab^{-1}$ is an idempotent, by assumption.
But then $ab^{-1} = bb^{-1}$ since $\mu$ is idempotent-separating,
which gives $ab^{-1}b = b$. 
But $\dom (a) = \dom (b)$ and so $a = b$.
\end{proof}

A topological space $X$ is said to be $T_{0}$\index{$T_{0}$-space} 
if for each pair of elements $x,y \in X$ there exists 
an open set which contains one but not both of $x$ and $y$. 
A {\em base} for a topological space 
is a set of open sets $\beta$ such that every open set of the topology
is a union of elements of $\beta$.
Let $X$ be an arbitrary set and $\beta$ a set of subsets of
$X$ whose union is $X$
and with the property that the intersection of any two elements
of $\beta$ is a union of elements of $\beta$. 
Then a topology can be defined on $X$ by defining 
the open sets to be the unions of elements of $\beta$.

As in Example~2.9, the inverse semigroup of all homeomorphisms between open subsets of $X$ is denoted by $\Gamma (X)$.
An inverse subsemigroup $S$ of $\Gamma (X)$ is said to be {\em topologically complete}  if the set-theoretic domains of the elements of $S$ form a base for the topology.

\begin{theorem} An inverse semigroup is fundamental if, and
only if, it is isomorphic to a topologically complete inverse
semigroup on a $T_{0}$-space.
\end{theorem}
\begin{proof}
Let $S$ be a fundamental inverse semigroup.   
We can assume by Theorem~4.6, 
that $S$ is a wide inverse subsemigroup of
a Munn semigroup $T_{E}$.
Put $\beta = \{e^{\downarrow} \colon \:  e \in E \}$.
Clearly, $E$ is the union of the elements of $\beta$,
and $\beta$ is closed under finite intersections. 
Thus $\beta$ is the base of a topology on the set $E$.
With respect to this topology, each element of $S$ is a homeomorphism between open subsets of $E$.
It remains to show that this topology is $T_{0}$.
Let $e,f \in E$ be distinct idempotents.
If $f \leq e$ then $f^{\downarrow}$ is an open set containing $f$ but not $e$.
If $f \not \leq e$ then $e^{\downarrow}$ is an open set containing $e$ but not $f$.
Thus the topology is $T_{0}$.

Conversely, let $S$ be a topologically complete 
inverse subsemigroup of the inverse semigroup $\Gamma (X)$ 
where the topology is $T_{0}$ and 
$\beta = \{\mbox{dom}\,\alpha \colon \:  \alpha \in S \}$
is a base for $\tau$.
We shall prove that $S$ is fundamental by showing that the centralizer
of the idempotents of $S$ contains only idempotents (Proposition~4.10).
Let $\phi \in S\setminus E(S)$.
Then there exists $x \in \mbox{dom}\,\phi$ such that $\phi (x) \neq x$,
because $\phi$ is not an idempotent.
Since $\tau$ is $T_{0}$, there exists an open set $U$ such that
$$\mbox{either } 
(\phi (x) \in U \mbox{ and } x \notin U)
\mbox{ or } 
(\phi (x) \notin U \mbox{ and } x \in U).$$
Since $\beta$ is a basis for the topology,
$U = \bigcup B_{i}$ for some $B_{i} \in \beta$.
It follows that there is a $B = B_{i} \in \beta$
such that
$$\mbox{either }
(\phi (x) \in B \mbox{ and } x \notin B)
\mbox{ or } 
(\phi (x) \notin B \mbox{ and } x \in B).$$
Observe that $1_{B} \in S$ since $B = \mbox{dom}\,\alpha$
for some $\alpha \in S$.
Thus the elements
$\phi 1_{B}$ and $1_{B} \phi$  belong to $S$.
In the first case, 
$\phi (x) \in B$ and $x \notin B$,
so that whereas $(\phi 1_{B})(x)$ is not defined, 
$(1_{B} \phi)(x)$ is defined.
Thus $\phi \notin Z(E(S))$.
In the second case, 
$(\phi 1_{B})(x)$ is defined 
and $(1_{B} \phi)(x)$ is not defined.
Thus once again $\phi \notin Z(E(S))$.
Hence in either case $\phi \notin Z(E(S))$.
\end{proof}


Let $S$ be an arbitrary inverse semigroup, 
let its image under the Munn representation be $T$,
and let $K$ be the centralizer of the idempotents of $S$.
Then $S$ is an extension of $K$ by $T$ where
the former is a presheaf of groups and the latter is a pseudogroup of transformations.

\begin{theorem}
Every inverse semigroup is an idempotent-separating extension of a presheaf of groups by a pseudogroup of transformations.
\end{theorem}

\subsection{Congruence-free inverse semigroups}

A useful application of fundamental inverse semigroups is in characterizing those semigroups which are congruence-free.
I shall concentrate only on the case of inverse semigroups with zero.
Douglas Munn once remarked to me that this was one of the few instances where the theory for inverse semigroups with zero was easier than it was for the one without.
We shall need a sequence of definitions before we can state our main result.

Although ideals are useful in semigroup theory, the connection between ideals and congruences is weaker for semigroups than it is for rings.
If $\rho$ is a congruence on a semigroup with zero $S$, 
then the set $I = \rho (0)$ is an ideal of $S$;  however, examples show that the congruence is not determined by this ideal.
Nevertheless, ideals can be used to construct some congruences on semigroups.
Let $I$ be an ideal in the semigroup $S$.
Define a relation $\rho_{I}$ on $S$ by:
$$(s,t) \in \rho_{I} \Leftrightarrow \mbox{either } s,t \in I 
\mbox{ or } s = t.$$
Then $\rho_{I}$ is a congruence.
The quotient semigroup $S/\rho_{I}$ is isomorphic to the set
$S \setminus I \cup \{ 0\}$ 
(we may assume that $0 \notin S \setminus I$) 
equipped with the following product: 
if $s,t \in S\setminus I$ then their product is $st$ if 
$st \in S \setminus I$, all other products are defined to be $0$.
Such quotients are called {\em Rees quotients}.

There is also a way of constructing congruences from subsets.
Let $S$ be a semigroup and let $L \subseteq S$.
Define a relation $\rho_{L}$ on $S$ by:
$$(s,t) \in \rho_{L} \Leftrightarrow 
(\forall a,b \in S)(asb \in L \Leftrightarrow atb \in L).$$
Then $\rho_{L}$ is a congruence on $S$, called the {\em syntactic congruence} of $L$.

An inverse semigroup with zero $S$ is said to be {\em $0$-simple}
if it contains at least one non-zero element and the only ideals are $\{0\}$ and $S$.
An inverse semigroup is said to be {\em congruence-free} if its only congruences are equality and the universal congruence.
Thus congruence-free-ness is much stronger than $0$-simplicity.
A congruence $\rho$ is said to be {\em $0$-restricted} if the $\rho$-class containing $0$ is just $0$.
Finally, define $\xi$ to be the syntactic congruence of the subset $\{ 0\}$.

\begin{lemma} 
The congruence $\xi$ is the maximum $0$-restricted congruence.
\end{lemma}
\begin{proof}
Let $\rho$ be a $0$-restricted congruence on $S$ and let $s \rho t$.
Suppose that $asb = 0$.
But $asb \xi atb$ and so since $\rho$ is $0$-restricted, we have that $atb = 0$.
By symmetry we deduce that $a \xi b$.
Thus $\rho \subseteq \xi$, as required.
\end{proof}

\begin{lemma} Let $S$ be an inverse semigroup with zero.
\begin{enumerate}

\item $\mu \subseteq \xi$.

\item The congruence $\xi$ restricted to $E(S)$ is the syntactic congruence determined by zero on $E(S)$.

\end{enumerate}
\end{lemma}
\begin{proof} 
(1) Let $s \mu t$. Suppose that $asb = 0$ then $asb \mu atb$ and so $atb = 0$.
By symmetry this shows that $s \xi t$.

(2) Let $e$ and $f$ be idempotents.
Suppose that for all idempotents $i$ we have that $ie = 0$ iff $if = 0$.
Let $aeb = 0$.
Then $a^{-1}aebb^{-1} = 0$.
Thus $a^{-1}a bb^{-1} e = 0$ and so $a^{-1}a bb^{-1} f = 0$.
Hence $a^{-1}a f bb^{-1} = 0$ and so $afb = 0$.
The reverse direction is proved similarly.
\end{proof}

An inverse semigroup with zero is said to be {\em $0$-disjunctive} if $\xi$ is the equality relation.

\begin{proposition} 
An inverse semigroup $S$ is $0$-disjunctive if and only if $E(S)$ is $0$-disjunctive and $S$ is fundamental.
\end{proposition}
\begin{proof}
If $S$ is $0$-disjunctive it follows by Lemma~4.14 that $E(S)$ is $0$-disjunctive and $S$ is fundamental.
Suppose that $E(S)$ is $0$-disjunctive and $S$ is fundamental.
Then $\xi$ restricted to $E(S)$ is the equality relation and so $\xi$ is idempotent-separating.
Thus by Lemma~4.4 $\xi \subseteq \mu$.
But $S$ is fundamental and so $\mu$ is the equality relation and so $\xi$ is the equality relation.
\end{proof}

\begin{lemma} Let $E$ be a meet semilattice with zero.
Then the following are equivalent.
\begin{enumerate}

\item $E$ is $0$-disjunctive.

\item For all distinct $e,f \in E$  nonzero there exists $g \in E$ such that
either $e \wedge g \neq 0$ and $f \wedge g = 0$ or $e \wedge g = 0$ and $f \wedge g \neq 0$.

\item  For all $0 \neq f < e$ there exists $0 \neq g \leq e$ such that $f \wedge g = 0$.

\end{enumerate}
\end{lemma}
\begin{proof} 

(1)$\Rightarrow$(2). This is immediate from the definition.

(2)$\Rightarrow$(3). Let $0 \neq f < e$.
Then there exists $g'$ such that $g' \wedge f = 0$ and $g' \wedge e \neq 0$
or $g' \wedge f \neq 0$ and $g' \wedge e = 0$.
Clearly the second case cannot occur.
Put $g = g' \wedge e$.
Then $g \leq e$, $g \neq 0$ and $g \wedge f = 0$, as required.

(3)$\Rightarrow$(1). Suppose that $e \xi f$ where $e$ and $f$ are both non-zero.
Then  $e \xi (e \wedge f)$ and so  $e \wedge f \neq 0$. 
Suppose that $e \wedge f \neq e$.
Then there exists $0 \neq g \leq e$ such that $(e \wedge f) \wedge g = 0$.
But clearly $e \wedge g \neq 0$.
We therefore have a contradiction and so $e \wedge f = e$.
Similarly $e \wedge f = f$ and so $e = f$, as required.
\end{proof}

We may now state the characterization of congruence-free inverse semigroups with zero.

\begin{theorem} 
An inverse semigroup with zero $S$ is congruence-free if and only if
$S$ is fundamental, $0$-simple and $E(S)$ is $0$-disjunctive.
\end{theorem}
\begin{proof}
Suppose that $S$ is congruence-free. 
Then $\mu$ is equality, 
there are no non-trivial ideals and $\xi$ is equality.
Thus $S$ is fundamental, $0$-simple and $E(S)$ is $0$-disjunctive.

To prove the converse, suppose that $S$ is fundamental, $0$-simple and $E(S)$ is $0$-disjunctive.
Let $\rho$ be a congruence on $S$ which is not the universal relation.
Then $\rho (0)$ is an ideal which is not $S$.
Thus it must be equal to $\{ 0\}$.
It follows that $\rho$ is a $0$-restricted congruence and so $\rho \subseteq \xi$.
But by Proposition~4.15, $\xi$ is the equality congruence and so $\rho$ is the equality congruence.
\end{proof}

The above theorem will be a useful criterion for congruence-free-ness once we have a nice characterization of $0$-simplicity.
This involves the one Green's relation we have yet to define.
Let $S$ be an inverse semigroup.
Define 
$$(s,t) \in \mathcal{J} \Leftrightarrow SsS = StS.$$
It is always true that $\mathcal{D} \subseteq \mathcal{J}$.
The meaning of the $\mathcal{J}$-relation for inverse semigroups is clarified by the following result.

\begin{lemma} Let $S$ be an inverse semigroup.
Then 
$a \in SbS$ if, and only if, there exists 
$u \in S$ such that $a\, \mathcal{D}  \,u \leq b$.
\end{lemma}
\begin{proof} Let $a \in SbS$.
Then $a = xby$ for some $x,y \in S$.  
By Proposition~2.23, 
there exist elements $x',y'$ and $b'$
such that $a = x'\cdot b'\cdot y'$ is a restricted product where
$x' \leq x, \, b' \leq b$ and $a' \leq a$.
Hence $a \, \mathcal{D}   \, b'$ which, 
together with $b' \leq b$,
gives $a\, \mathcal{D}  \,b' \leq b$.
Conversely, suppose that $a \, \mathcal{D}  \, b' \leq b$. 
From $a \, \mathcal{D}  \, b'$ we have that $a \, \mathcal{J}  \,b'$, 
and from $b' \leq b$ we have that 
$Sb'S \subseteq SbS$.  
Thus $a \in  SbS$.
\end{proof}

\begin{lemma} Let $S$ be an inverse semigroup with zero.
 Then it is $0$-simple if, and only if, $S \neq \{ 0\}$  and the only $\mathcal{J}$-classes  are $\{0 \}$ and $S\setminus \{0 \}$. 
\end{lemma}
\begin{proof}
Let $S$ be $0$-simple and let $s,t \in S$ be a pair of non-zero elements.
Both $SsS$ and $StS$ are ideals of $S$ and so must be equal.
Thus $(s,t) \in  \mathcal{J}$.
Conversely, suppose that the only non-zero $\mathcal{J}$-class is $S\setminus\{ 0\}$.
Let $I$ be any non-zero ideal of $S$.
Let $s \in I$ and $t \in S$ be non-zero elements.
By assumption, $(s,t) \in \mathcal{J}$.
Thus $t = asb$ for some $a,b \in S$ and so $t \in I$.
Hence $I = S \setminus \{ 0\}$.
\end{proof}

\begin{proposition} Let $S$ be an inverse semigroup with zero.
\begin{enumerate}

\item $S$ is 0-simple if, and only if, for any two non-zero 
elements $s$ and $t$ in $S$ there exists an element $s'$ such that
$s\, \mathcal{D}  \,s' \leq t$.

\item $S$ is 0-simple if, and only if, for any two non-zero 
idempotents $e$ and $f$ in $S$ there exists an idempotent $i$ such that
$e\, \mathcal{D}  \,i \leq f$. 

\end{enumerate}
\end{proposition}
\begin{proof} 

(1) By Lemma~4.19, an inverse semigroup is 0-simple
if it consists of exactly two $\mathcal{J}$-class $\{0 \}$ and $S \setminus \{0 \}$.
Thus any two non-zero elements of $S$ are $\mathcal{J}$-related.
The result is now immediate by Lemma~4.18.

(2) Suppose the condition on the idempotents holds. 
Let $s,t \in S$ be a pair of non-zero elements.
Then $e = ss^{-1}$ and $f = tt^{-1}$ are non-zero idempotents and so, by assumption,
there is an idem\-pot\-ent $i$ such that $e\,\mathcal{D}  \,i \leq f$.
Put $u = it$.
Then $u \leq t$, and 
$uu^{-1} = it(it)^{-1} = itt^{-1} = if = i$.
Thus $s\, \mathcal{D}  \,u \leq t$.
The proof of the converse is straightforward.
\end{proof}

\section{Transitive representations}

There are two basic definitions of `action' for an inverse semigroup.
If we regard an inverse semigroup as just a semigroup then there is the usual notion of a {\em left $S$-set}.
However, just as the Cayley theorem in group theory, 
the Wagner-Preston theorem motivates another class of actions.
It is this notion that we shall study in this section.
All inverse semigroups will have a zero and we shall assume that all homomorphisms preserve the zero.
The theory developed is remarkably similar to the classical theory of transitive representations of groups.

A {\em representation} of an inverse semigroup by means of partial bijections 
is a homomorphism $\theta \colon \: S \rightarrow I(X)$ to the symmetric inverse monoid on a set $X$.

A representation of an inverse semigroup in this sense leads to a corresponding notion of an action of the inverse semigroup $S$ on the set $X$:
the associated action is defined by $s \cdot x = \theta (s)(x)$, if $x$ belongs to the set-theoretic domain of $\theta (s)$.
The action is therefore a partial function from $S \times X$ to $X$ mapping $(s,x)$ to $s \cdot x$ when $\exists s \cdot x$
satisfying the two axioms:
\begin{description}
\item[{\rm (A1)}] If $\exists e \cdot x$ where $e$ is an idempotent then $e \cdot x = x$.

\item[{\rm (A2)}] $\exists (st) \cdot x$ iff $\exists s \cdot (t \cdot x)$ in which case they are equal.
\end{description}

Representations and actions are different ways of describing the same thing.
For convenience, we shall use the words `action' and `representation' interchangeably:
if we say the inverse semigroup $S$ acts on a set $X$ then this will imply the existence
of an appropriate homomorphism from $S$ to $I(X)$.
If $S$ acts on $X$ we shall often refer to $X$ as a {\em space} and its elements as {\em points}.
A subset $Y \subseteq X$ closed under the action is called a {\em subspace}.
Disjoint unions of actions are again actions.
The proof of the following is straightforward.

\begin{lemma} Let $S$ act on $X$.
Define a relation $\sim$ on $X$ by $x \sim y$ iff there exists $s \in S$
such that $\exists s \cdot x$ and $s \cdot x = y$.
This relation is symmetric and transitive.
It is reflexive if and only if 
for each $x \in X$ there is $s \in S$ such that $\exists s \cdot x$.
\end{lemma}

\begin{remark}{\em 
An action satisfying the condition above is said to be {\em effective}.
From now on, we shall regard effectiveness as part of the definition of an inverse semigroup action;
if an action were not effective, then we could restrict our attention to the largest subset of $X$ where it was.}
\end{remark}

The action of an inverse semigroup $S$ on the set $X$
induces an equivalence relation $\sim$ on the set $X$ when we define
$x \sim y$ iff $s \cdot x = y$ for some $s \in S$.
The action is said to be {\em transitive} if $\sim$ is $X \times X$.
Just as in the theory of permutation representations of groups, every representation of an inverse semigroup is a disjoint union of transitive representations.

Transitive actions of inverse semigroups are characterized by special kinds of inverse semigroups in a way generalizing the
relationship between transitive group actions and subgroups.
Fix a point $x \in X$, and consider the set $S_{x}$ consisting of all
$s \in S$ such that $s \cdot x = x$.
We call $S_{x}$ the {\em stabilizer} of the point $x$.
If an element $s$ fixes a point then so too will any element above $s$,
and so the set $S_{x}$ is a closed inverse subsemigroup of $S$.
Observe that stabilizers cannot contain zero.
Now let $y \in X$ be any point.
By transitivity, there is an element $s \in S$ such that $s \cdot x = y$.
Observe that because $s \cdot x$ is defined so too is $s^{-1}s$ and that $s^{-1}s \in S_{x}$.
An easy calculation shows that $[sS_{x}]$ is the set of all elements of $S$ which map $x$ to $y$.

A closed inverse subsemigroup of $S$ that does not contain zero is said to be {\em proper}.
Let $H$ be a proper closed inverse subsemigroup of $S$.
Define a {\em left coset} of $H$ to be a set of the form $(sH)^{\uparrow}$ where $s^{-1}s \in H$.
The following are well-known but we include the proofs for the sake of completeness.

\begin{lemma} Let $H$ be a proper closed inverse subsemigroup of $S$.

\begin{enumerate}

\item Two cosets $(sH)^{\uparrow}$ and $(tH)^{\uparrow}$ are equal iff $s^{-1}t \in H$.

\item If $(sH)^{\uparrow} \cap (tH)^{\uparrow} \neq \emptyset$ then $(sH)^{\uparrow} = (tH)^{\uparrow}$.

\end{enumerate}
\end{lemma}
\begin{proof} 
(1) Suppose that  $(sH)^{\uparrow} = (tH)^{\uparrow}$.
Then $t \in (sH)^{\uparrow}$ and so $sh \leq t$ for some $h \in H$.
Thus $s^{-1}sh \leq s^{-1}t$.
But $s^{-1}sh \in H$ and $H$ is closed and so $s^{-1}t \in H$.

Conversely, suppose that $s^{-1}t \in H$.
Then $s^{-1}t = h$ for some $h \in H$ and so
$sh = ss^{-1}t \leq t$.
It follows that $tH \subseteq sH$ and so $(tH)^{\uparrow} \subseteq (sH)^{\uparrow}$.
The reverse inclusion follows from the fact that $t^{-1}s \in H$ since
$H$ is closed under inverses.

(2) Suppose that $a \in (sH)^{\uparrow} \cap (tH)^{\uparrow}$. 
Then $sh_{1} \leq a$ and $th_{2} \leq a$ for some $h_{1},h_{2} \in H$.
Thus $s^{-1}sh_{1} \leq s^{-1}a$ and $t^{-1}th_{2} \leq t^{-1}a$.
Hence $s^{-1}a,t^{-1}a \in H$.
It follows that $s^{-1}aa^{-1}t \in H$,
but $s^{-1}aa^{-1}t \leq s^{-1}t$.
This gives the result by (i) above.
\end{proof}

We denote by $S/H$ the set of all left cosets of $H$ in $S$.
The inverse semigroup $S$ acts on the set $S/H$ when we define
$$a \cdot (sH)^{\uparrow} = (asH)^{\uparrow} \Leftrightarrow \dom (as) \in H.$$
This defines a transitive action.

Let $S$ be an inverse semigroup acting on the sets $X$ and $Y$.
A bijective function $\alpha \colon \: X \rightarrow Y$
is said to be an {\em equivalence} from $X$ to $Y$ if
$\exists s \cdot x \Leftrightarrow \exists s \cdot \alpha (x)$
and if either side exists we have that
$\alpha (s \cdot x) = s \cdot \alpha (x)$.
As with group actions, equivalent actions are the same except for the labelling of the points.
The proof of the following theorem is a straightforward generalization of the one for groups.

\begin{theorem} Let $S$ act transitively on the set $X$.
Then the action is equivalent to the action of $S$ on the set $S/S_{x}$ where $x$ is any point of $X$.\qed
\end{theorem}

If $H$ and $K$ are any closed inverse subsemigroups of $S$ that do not contain zero then they determine
equivalent actions if and only if there exists $s \in S$ such that
$$sHs^{-1} \subseteq K \mbox{ and } s^{-1}Ks \subseteq H.$$
This relationship between two closed inverse subsemigroups is called {\em conjugacy}
although it is important to observe that equality need not hold in the definition above.

\begin{lemma} If $H$ and $K$ are conjugate as above then $ss^{-1} \in K$ and $s^{-1}s \in H$.
Also $(sHs^{-1})^{\uparrow} = K$ and $(s^{-1}Ks)^{\uparrow} = H$.
\end{lemma}
\begin{proof}
Let $e \in H$ be any idempotent.
Then $ses^{-1} \in K$.
But $ses^{-1} \leq ss^{-1}$ and so $ss^{-1} \in K$.
Similarly $s^{-1}s \in H$.

We have that $sHs^{-1} \subseteq K$ and so $(sHs^{-1})^{\uparrow} \subseteq K$.
Let $k \in K$.
Then $s^{-1}ks \in H$ and $s(s^{-1}ks)s^{-1} \in sHs^{-1}$ and
$s(s^{-1}ks)s^{-1} \leq k$.
Thus  $(sHs^{-1})^{\uparrow} = K$, as required.
\end{proof}

Thus to study transitive actions of an inverse semigroups with zero $S$ it is enough to study up to conjugacy 
the closed inverse subsemigroups of $S$ not containing zero.

\section{Notes on Chapter 1}

I have assumed the reader is familiar with the basics of semigroup theory such as could be gleaned from the first few sections of Howie \cite{Howie}.
There is a mild use of category theory for which the standard reference is Mac~Lane \cite{Maclane}.
There are currently two books entirely devoted to inverse semigroup theory:
Petrich's \cite{Petrich} and mine \cite{Lawson}.
Petrich's book is pretty comprehensive up to 1984 and is still a useful reference.
Its only drawback is the poor index which makes finding particular topics a bit of a chore.
My book is less ambitious.
Its goal is to motivate the study of inverse semigroups by concentrating on concrete examples and was completed in 1998.
In writing this chapter, I have drawn mainly upon my own book  
but, in the case of the section on congruence-free inverse semigroups,
I have based my discussion on Petrich with some flourishes of my own.
I have only touched on the history of inverse semigroup theory here because I did that in great detail \cite{Lawson}.


Inverse semigroups are special kinds of regular semigroups and arbitrary regular semigroups are also interesting and important.
The deepest work in general regular semigroup theory has been carried out by K.~S.~S.~Nambooripad \cite{N1,N2}.
The 1970's and 80's seemed to be halcyon days for regular semigroup theory.
Howie's book is still heavily biased in their favour and
many results in this chapter are really special cases of results for general regular semigroups.
However, in recent years regular semigroup theory has started to re-emerge and to connect with other parts of mathematics.
This wider appreciation of regular semigroups is due in large measure to a paper by Kenneth S. Brown \cite{B}
who showed that a class of idempotent semigroups was useful in understanding random walks on certain groups.
For a recent development of this line of work see \cite{MS}.

My passing reference to Girard's work in linear logic prior to Lemma~2.6 can be verified by checking out the third bullet-point on page~345 of \cite{GLR}.

Sheaves of groups have important applications in mathematics \cite{Iversen}.
The cohomology of inverse semigroups was introduced by Lausch \cite{L} and put into the correct categorical framework by Loganathan in his remarkable paper \cite{Log}. 
A cohomology of inverse semigroups was also introduced by Renault \cite{R}.

The sense in which an inverse semigroup is an extension of a presheaf of groups by a pseudogroup requires non-abelian cohomology.
A theory of such extensions generalizing the classical group case was worked out by Coudron \cite{C} and D'Alarcao \cite{D}.
More on extensions of inverse semigroups can be found in Chapter~5 of \cite{Lawson}.

Much of what can be said about inverse semigroups can be generalized easily to inverse categories and there are good reasons for doing so.
If $S$ is an inverse semigroup then its {\em Cauchy completion} is an inverse category.
The Cauchy completion of $S$ is the category $C(S)$ whose elements are of the form $(e,s,f)$ where $esf = s$ with the obvious partial binary operation.
Cauchy completions are important in the Morita theory of inverse semigroups.
The earliest reference to inverse categories I know is \cite{Kastl}.
They have been used by Grandis \cite{Grandis} in developing a theory of local structure that parallels Ehresmann's.
The minimum group congruence can be generalized to inverse categories to yield the minimum groupoid congruence.
There are two nice applications of this.
The first is in constructing a topological groupoid from a pseudogroup of transformations;
this is described on pages 63 and 64 of my book \cite{Lawson}.
The second comes from group theory.
Let $G$ and $H$ be groups.
Then an {\em almost (or virtual) isomorphism} from $G$ to $H$ is an isomorphism $\alpha \colon A \rightarrow B$
where $A$ is a subgroup of finite index in $G$ and $B$ is a subgroup of finite index in $H$.
The intersection of a finite number of subgroups of finite index is again of finite index.
Thus groups and almost isomorphisms form an inverse category.
The {\em abstract commensurator} of a group $G$ is then the maximum group image of the local inverse monoid at $G$ \cite{BB}.


\end{document}